\documentclass[11pt,a4paper,reqno]{amsart}

\usepackage[T1]{fontenc}
\usepackage[utf8]{inputenc}
\usepackage[left=2cm,right=2cm,top=2cm,bottom=2cm]{geometry}

\usepackage{mathtools}
\mathtoolsset{showonlyrefs}
\usepackage{amssymb}
\usepackage{dsfont}

\usepackage{enumitem}

\usepackage[square,sort,comma,numbers]{natbib}
\usepackage[colorlinks=true,linkcolor=blue,citecolor=blue,urlcolor=blue,breaklinks]{hyperref}
\usepackage{cleveref}

\numberwithin{equation}{section}
\numberwithin{figure}{section}

\newcommand{\cL}{\mathcal{L}}
\newcommand{\sfL}{\mathsf{L}}
\newcommand{\sfT}{\mathsf{T}}

\newcommand{\R}{\mathbb{R}}
\newcommand{\Rd}{\R^d}
\newcommand{\RRd}{\R^d\times\R^d}
\newcommand{\dd}{\,\mathrm d}
\newcommand{\cLs}{\mathcal L^\star}

\newcommand{\one}{\mathbf 1}

\newcommand{\eps}{\varepsilon}
\def\1{\mathds{1}}
\newcommand{\norm}[1]{\left\lVert #1\right\rVert}
\newcommand{\abs}[1]{\left\lvert #1\right\rvert}
\newcommand{\cK}{\mathcal K}
\newcommand{\cM}{\mathcal M}

\newcommand{\wdot}{\lfloor \cdot \rceil}
\newcommand{\wx}{\lfloor x \rceil}

\newcommand{\wv }{\lfloor v \rceil}
\renewcommand{\wr}{\lfloor r \rceil}
\newcommand{\ws}{\lfloor s \rceil}
\newcommand{\wu}{\lfloor u \rceil}

\theoremstyle{plain}
\newtheorem{theorem}{Theorem}[section]
\newtheorem{lemma}[theorem]{Lemma}
\newtheorem{hypothesis}[theorem]{Hypothesis}
\newtheorem{definition}[theorem]{Definition}

\newtheorem{proposition}[theorem]{Proposition}

\AddToHook{env/lemma/begin}{\crefalias{theorem}{lemma}}
\AddToHook{env/hypothesis/begin}{\crefalias{theorem}{hypothesis}}
\AddToHook{env/proposition/begin}{\crefalias{theorem}{proposition}}
\AddToHook{env/definition/begin}{\crefalias{theorem}{definition}}
\AddToHook{env/remark/begin}{\crefalias{theorem}{remark}}
\AddToHook{env/assumption/begin}{\crefalias{theorem}{assumption}}

\title[Kinetic Fokker-Planck equations with weak space confinements]{Lyapunov estimates for non-factorised kinetic Fokker-Planck equations under weak spatial confinements}

\author{\'Emeric Bouin}
\address[E. Bouin]{CEREMADE - Université Paris-Dauphine, PSL Research University, UMR CNRS 7534, Place du Mar\'echal de Lattre de Tassigny, 75775 Paris Cedex 16, France.}
\email{bouin@ceremade.dauphine.fr}

\date{\today}

\begin{document}
    \begin{abstract}
 In this paper, we construct weak Lyapunov functionals for kinetic Fokker-Planck equations with weak external confinements. This problem has remained widely open in the existing literature. Apart from showing interesting structures from the models, deriving these Foster-Lyapunov weights allows us to use Harris type arguments to show existence of steady states for these equations, and convergence to equilibrium with an explicit stretched exponential rate. 

    \end{abstract}
\maketitle
    

\section{Introduction}
This work is devoted to the kinetic Fokker-Planck equation,
\begin{equation}\label{eq:mainFP}
    \partial_t f=\mathcal{L}f:=-v\cdot\nabla_x  f+\nabla_x  V\cdot \nabla_v  f+\nabla_v \cdot \left( \nabla_vf + \wv^{\beta-2}v f\right) 
\end{equation}
where the unknown function $f=f(t,x,v)$ depends on time $t\in[0,\infty)$, position $x\in\Rd$ and velocity $v\in\Rd$. In this paper, $\beta >0$. We consider external potentials $V\colon \Rd\to\R$ of the form
\begin{equation}\label{eq:V}
V=\frac{\wdot^\alpha}{\alpha}
\end{equation}
with $0 < \alpha\leq 1$, where we denote $\wdot=\left(1+\abs{\cdot}^2\right)^{\frac12}$. The arguments extend to potentials with the same asymptotic behaviour as $\frac{\wdot^\alpha}{\alpha}$;
for readability, however, we restrict the presentation to the exact situation \eqref{eq:V}. The paper could be written equally with a potential $V$ that is behaving like $\frac{\wx^\alpha}{\alpha}$, without strict equality, but we stick to this to ease readability. For such an equation, the local equilibrium is of the form
\begin{align}\label{eq:Mexp}
     \cM=c_\beta^{-1} e^{-\frac{\wdot^\beta}{\beta}},
\end{align}
with ($\beta>0$ and) $c_\beta = \int_{\R^d} \exp\left(-\frac{\wv^\beta}{\beta}\right)\dd v $. 

The operator $\cL$ in \eqref{eq:mainFP} can be split into $\cL=\sfL - \sfT$, where $\sfT$ is a transport operator and $\sfL$ is a collision operator. The transport operator $\sfT$ is 
\begin{equation*}
\sfT f = v\cdot\nabla_x  f-\nabla_x  V\cdot \nabla_v  f.
\end{equation*}
The collision operator $\sfL$ is the Fokker-Planck operator,
\begin{equation*}
\sfL f =\nabla_v \cdot \left( \cM\nabla_v \left(\frac{f}{ \cM}\right)\right) = \Delta_v f +\nabla_v \cdot\left( \wv ^{\beta -2} v \, f\right).
\end{equation*}
The case $\beta=2$, for any $\alpha >0$, is very special, since in this case the equation has an explicit global Gibbs state,
\begin{equation*}
G \propto \exp(-E),
\end{equation*}
where $E$ is the mechanical energy associated with the Hamiltonian transport operator $\sfT$,
\begin{equation}
    E(x,v)=\frac{\abs{v}^2}{2} + V(x).
\end{equation}
This equation has been extensively studied and various hypocoercivity methods have been used \cite{Dolbeault2015,MR2576899,MR2813582,Helffer2005,Hrau2004,MR2294477,Villani2009} to show exponential convergence to the steady state in weighted $\sfL^2$ norms. Still with $\beta=2$, in \cite{Cao2019}, the author proves convergence for weaker potential \eqref{eq:V} with $\alpha\in(0,1)$. We will come back to this later on.
For other results on weak confinement we refer the interested reader to \cite{Bakry2008,MR2499863}. When $\alpha > 1$, together with Ziviani, the author has extended in \cite{BouinZiviani} the results of \cite{C21}, where only $\beta\geq2$ is covered. Result for the missing range of parameters $0<\beta<2$ are provided, and stretched exponential decay rates are observed. Nevertheless, the approaches described in \cite{BouinZiviani} were crucially using the fact that $\alpha > 1$ to gain spatial coercivity. One aim of this paper is to go beyond this obstruction and tackle weak spatial confinements.  

On the other hand, when the local equilibrium $\cM$ is not a Gaussian, the transport operator and the collision operator do not share the same kernel. This fact makes the existence of a stationary state $G$ much harder to prove. It is therefore necessary to use non-explicit methods, like the Harris theorem, to show the existence of such an equilibrium. In particular, this strategy requires the knowledge of a Lyapunov type condition, which is in general quite complicated to get. 

For further use, we define the formal dual operator $\cL^{\star}$ by
\begin{equation}\label{Lstar}
    \mathcal{L}^{\star}m:=v\cdot\nabla_x  m-\nabla_x  V\cdot \nabla_v  m +\Delta_v m -\wv^{\beta-2}v \cdot\nabla_v  m  
\end{equation}
We write
\begin{equation*}
\sfL^\star:=\Delta_v-\wv ^{\beta-2}v\cdot\nabla_v,
\end{equation*}
so that $\cL^\star=\sfL^\star + \sfT$, since $\sfT$ is anti-symmetric.

For a positive weight function $m\colon \RRd\to \R$, we denote by $\sfL^1(m)$ the functional space defined by the norm
\begin{equation*}
\norm{f}_{\sfL^1(m)}:=\iint_{\R^d \times \R^d} \abs{f(x,v)}m(x,v)\dd x \dd v .
\end{equation*}

The first main results are the construction of weak Foster-Lyapunov functionals for $\alpha \in (0,1]$ and any $\beta >0$. This was widely open from the previous literature.

\begin{theorem}\label{thm:main-lyapunov}
Let $d \geq 1$, $\alpha \in (0,1]$ and $\beta >0$.
\begin{enumerate}
    \item For $\alpha =1 $ and $\beta >0$. For $0<\theta\le \min\left(1,\frac{\beta}{2}\right)$, there exists an explicit coercive $m_\theta$ such that
\begin{equation*}
  \mathcal L^{\star} m_\theta
  \le C\mathbf 1_{\mathcal K} -c\,\frac{m_\theta}{(\ln(e+m_\theta))^{\frac{1-\theta}{\theta}}}.
\end{equation*}
When $\theta =1$, the denominator is really interpreted as $1$.
\item  For $\alpha <1 $ and $\beta \geq 2$. There exists an explicit positive and coercive weight $m$ such that
\begin{equation*}
 \cL^\star m
 \le C\one_{\mathcal{K}}
 - c
  \frac{m}{(\ln(e+m))^{\frac{2(1-\alpha)}{\alpha}}},
\end{equation*}
\item  For $\alpha <1 $ and $\beta <2$.
There exists an explicit positive and coercive weight $m$ such that
\begin{equation*}
 \cL^\star m
 \le C\one_{\mathcal{K}}
 - c
  \frac{m}{(\ln(e+m))^{\sigma}},
\end{equation*}
with $\sigma = \frac{2}{\beta}\left(1-\frac{\beta}{2}+\frac{2(1-\alpha)}{\alpha}\right).$ 
\end{enumerate}
Above, $\mathcal{K}$ denotes a compact in the phase space $\R^d \times \R^d$, and $C$ and $c$ are explicit constants (different from one item to another).
\end{theorem}

Let us comment on the proof of this result. For the first item, we will separate the cases $\beta > 1$ and $\beta \leq 1$. Indeed, the latter can be dealt by adapting and extending a bit the construction made in \cite{BouinZiviani} for $\alpha >1$ and does not really require much technology (however, it would be possible to adapt the constructions made for $\alpha < 1$ so that they also work for $\alpha =1$ with a similar shape, but a bit more involved than the one we write here). The construction made for the former is a bit different, and is an opportunity to introduce the so-called velocity cell problem, which is,  \begin{equation*}
 \sfL^\star\chi=-v\qquad\text{in }\R^d.
\end{equation*}
supplemented with a normalisation condition, 
\begin{equation*}\int_{\R^d}\chi(v)\,\cM(v) \dd v=0.
\end{equation*}
Actually, since the construction and properties of $\chi$ are necessary for the two other items, we dedicate the full \Cref{sec:corrector} to it. This section can be read independently and is interesting on its own. Moreover, the proofs in the case $\alpha =1$ show the importance of using velocity cut-offs in the expression of the Lyapunov functional. There, the radius of the cut-off in velocity will not depend on space. 

For the two other items, we create functionals containing a cut-off in velocity whose radius depends on space ($\beta > 2$) or on the full energy ($\beta \leq 2$). Even though some existing results in the case $\beta = 2$ were present in \cite{Cao2019}, we provide a construction here that improves theirs, since the sharp value $\frac{2(1-\alpha)}{\alpha}$ is reachable in our construction. We dedicate a short section to this special case, since it is easier to write than the general $\beta < 2$.  

Although the Lyapunov functions constructed below take different forms
in the various parameter regimes, they all rely on the same two
mechanisms. The first one transfers the velocity dissipation to the
space variable through a suitable kinetic corrector. The second one
uses the direct coercivity of the velocity Fokker-Planck operator when
the velocity is sufficiently large. Let us briefly present both of these.

The mechanical energy $E$ is naturally adapted to the Hamiltonian transport, since
$\mathsf T E=0$. However,
\begin{equation*}
\mathsf L^\star E
   =d-\lfloor v\rceil^{\beta-2}|v|^2
\end{equation*}
only provides a useful negative drift at large velocities. In
particular, a function of the energy alone does not produce any return
mechanism in the region where $|x|$ is large and $|v|$ remains
moderate. This is a major difficulty. In contrast, a purely spatial function $U(x)$ produces the
non-signed transport term $v\cdot\nabla U$. A Lyapunov function must
therefore couple the space and velocity variables, especially to deal with crossed terms. This is actually reminiscent from the hypocoercive structure of the kinetic operator $\cL$.

An interesting way to cancel a term of the form $v\cdot\nabla U$ is provided by the velocity cell problem
\begin{equation*}
\mathsf L^\star\chi=-v,
\end{equation*}
that would naturally appear when doing macroscopic limits of \eqref{eq:mainFP}. Indeed, for a smooth spatial profile $U$, define
\begin{equation*}
F(x,v):=U(x)+\nabla U(x)\cdot\chi(v).
\end{equation*}
Thanks to the following crucial cancellation, 
\begin{equation*}
v\cdot\nabla U+\nabla U\cdot\mathsf L^\star\chi=0,
\end{equation*}
one may see that
\begin{equation*}
\mathcal L^\star F
 =
v^\top D_x^2U\,\chi
-
\nabla V^\top D_v\chi\,\nabla U.
\end{equation*}
If $U=\Psi(V)$, the second term becomes
\begin{equation*}
-\Psi'(V)\,
 \nabla V^\top D_v\chi\,\nabla V.
\end{equation*}
A positivity estimate on $D_v\chi$ therefore turns the velocity
dissipation into a good negative term, whereas the term containing
$D_x^2U$ is treated as a lower-order remainder. This identity explains
why the asymptotic bounds on both $\chi$ and $D_v\chi$ presented in \Cref{sec:corrector} are essential
for all subsequent constructions.

In our constructions, the corrector cannot in general be used globally in velocity. Its
growth may destroy needed coercivity properties of the Lyapunov function or generate error terms
which are larger than the available dissipation. We therefore repeatedly separate
the small-to-moderate velocity region, where the corrected spatial phase is
effective, from the high-velocity region, where the energy itself is
dissipated by the collision operator. The different constructions
below correspond to different ways of gluing these two mechanisms.

When $\alpha=1$, the quantity $|\nabla V|$ remains of order one at
spatial infinity. A fixed velocity threshold is therefore sufficient.
For $\beta\leq1$, the elementary truncated cross term
$\wx^{-1} x\cdot v$ already creates the missing negative spatial
contribution. For $\beta>1$, we use the cell corrector and combine a
corrected spatial weight with an energy weight. 

The situation is different when $0<\alpha<1$, since
\begin{equation*}
|\nabla V(x)|^2
 \asymp
 \lfloor x\rceil^{-2(1-\alpha)}
 \asymp
 V(x)^{-\frac{2(1-\alpha)}{\alpha}}
\end{equation*}
vanishes at infinity. The velocity scale at which the direct collision
dissipation becomes stronger than the corrected spatial dissipation
must then depend on the position, or equivalently on the total energy.
For $\beta>2$, this leads to a position-dependent threshold
$R(\tau)$, where $\tau=|\nabla V|^2$ at leading order. The auxiliary
kinetic energy is multiplied by a coefficient which is of size $\tau$
at moderate velocities and becomes equal to one at large velocities. For $\beta<2$, we instead use as threshold criterion,
\begin{equation*}
\Theta(x,v)=\frac{\lfloor v\rceil}{E(x,v)^a}.
\end{equation*}
When $\Theta$ is small, one has $E\asymp V$ and the corrected spatial
phase
\begin{equation*}
F=V^{\frac{\beta}{2}}+\nabla(V^{\frac{\beta}{2}})\cdot\chi
\end{equation*}
is appropriate. When $\Theta$ is large, $v$ is large and the phase $E^{\frac{\beta}{2}}$ is
directly dissipated by the velocity operator. We interpolate between
these two phases by means of a function $\zeta(\Theta)$ which is flat
at the origin. 

The logarithmic powers in \Cref{thm:main-lyapunov} directly reflect the degeneration
of the spatial drift. For instance, when $\alpha<1$ and $\beta\geq2$,
the phase is of order $V$ and the drift is of order
$|\nabla V|^2\asymp V^{-\frac{2(1-\alpha)}{\alpha}}$, which yields the exponent
$\frac{2(1-\alpha)}{\alpha}$. When $\beta<2$, the phase is of order
$V^{\frac{\beta}{2}}$ and the spatial drift is of order
\begin{equation*}
V^{\frac{\beta}{2}-1}|\nabla V|^2,
\end{equation*}
which gives
\begin{equation*}
\sigma
 =
\frac{
1-\frac{\beta}{2}
+\frac{2(1-\alpha)}{\alpha}
}{
\frac{\beta}{2}
}.
\end{equation*}

The evolution equation is supplemented with an initial datum $f_0\in\sfL^1$. The equation is mass conservative, that is,
\begin{equation*}
\int_{\R^d \times \R^d} f\dd x \dd v  = \int_{\R^d \times \R^d} f_0\dd x \dd v ,\qquad\text{for all } t\geq 0.
\end{equation*}
A next important result is existence and uniqueness of a steady state $G$ to \eqref{eq:mainFP} and convergence of solutions towards $G$. In a a situation where $G$ is not likely to be explicit, an important approach is the Harris theorem. The seminal works for such results are due to Doeblin \cite{Doeblin1940} and Harris \cite{H56}. Doeblin established exponential convergence for Markov processes under the assumption of uniformly positive transition probabilities (the so-called Doeblin theorem \cite{CM21,EY23}), while Harris provided sufficient conditions for the existence of a unique stationary measure. We refer to \cite{BouinZiviani} for a review of recent or not-so-recent bibliography about applications and development of Harris type theorems to partial differential equations. 

The main hypotheses of the Harris theorem are the following.
\begin{hypothesis}[Weak Lyapunov condition]\label{Hyp:Lyapunov}
There exists a continuous function $m\colon \R^d\times \R^d \to [1,+\infty)$ with pre-compact level sets such that 
\begin{equation}
\cL^{\star} m\leq C\1_{B_R} -\phi(m) 
\end{equation}
for some constants $C,R>0$ and some strictly concave function $\phi\colon\R_+\to\R$ with $\phi(0)=0$ and increasing to infinity.
\end{hypothesis}

\begin{hypothesis}[Positivity condition]\label{Hyp:minorisation}
We say that the stochastic semigroup $S_\cL$ satisfies the positivity condition on a set $\mathcal{C}$ if there exists a probability measure $\mu_*$ and a constant $\eta\in(0,1)$  such that for a certain $T>0$
\begin{equation}\label{con:minorisation}
S_{\cL}(T)\mu \geq \eta\mu_*\int_\mathcal{C}\mu 
\end{equation}
for all positive measures $\mu$.
\end{hypothesis}

We now give the statement of the Harris theorem we shall use and we refer to \cite[Thm 5.6]{CM21} for the details of the proof.
\begin{theorem}[Sub-geometric Harris Theorem] \label{thm:harris}
Consider a stochastic semigroup $S_\cL$ with generator $\cL$ that satisfies \Cref{Hyp:Lyapunov} for a continuous function $m\colon \RRd\to[1,\infty)$ and \Cref{Hyp:minorisation} in a set $\mathcal{C}=\{(x,v)\in\R^d\times\R^d\;\colon\; m(x,v)\leq M \}$ for large enough $M$. Then there exists a unique invariant measure $\mu_G\in \mathcal{P}(\R^d\times\R^d)$ such that
\begin{equation}\label{Harris:L1}
    \int_{\R^d\times\R^d}\phi(m)d\mu_G<\infty
\end{equation}
and there exists a decay-rate function $\Theta(t)$ such that
\begin{equation}\label{Harris:rate}
    \norm{S_{\cL}(t)\mu-\mu_G}_{TV}\lesssim \Theta(t)\footnote{We use the notation $\mathsf a\lesssim \mathsf b$ if there exist a constant $\mathsf c>0$ such that $\mathsf a\leq \mathsf{c}\,\mathsf{b}$.}\norm{\mu-\mu_G}_m
\end{equation}
for any probability measure $\mu$. 
\end{theorem}
The function $\Theta(t)$ can be explicitly computed from the concave function $\phi$ appearing in the Lyapunov condition, for all the details we refer to \cite[Sec. 4]{CM21}. 
In this work, since $\phi(m)=m/(\ln(e+m)^\sigma$ for some $\sigma>0$, \cite[Sec. 4]{CM21} tells that the function $\Theta(t)$ is given by
\begin{equation*}
\Theta(t)=e^{-\lambda t^{\frac{1}{1+\sigma}}}
\end{equation*}
for an explicitly computable constant $\lambda>0$. Observe that formally, the function $m \mapsto \phi(m)=m/(\ln(e+m))^\sigma$ is not \textit{everywhere} concave. However, it is concave for $m$ large, and this does not change anything to the conclusions since one can modify $\phi$ so as to make is strictly concave up to modifying the compact in the Lyapunov statement. In order to keep expressions simple we keep this as it is. The positivity condition of \Cref{Hyp:minorisation} has been derived through Harnack-type estimates in \cite[Section 3]{BouinZiviani}; we thus shall refer to this since the proof there only uses the local boundedness of $\nabla_x V - \wv^{\beta-2} v$, which is true here. A consequence of \Cref{thm:main-lyapunov} is the following result.

\begin{theorem}\label{thm:main-Harris}
 For any given $\alpha \in (0,1]$ and $\beta >0$, take $m$ given by \Cref{thm:main-lyapunov}. There exists a positive normalised steady state $G\in \sfL^1\left(\phi(m) \right)$. Moreover, let $f$ be a solution to \eqref{eq:mainFP}, with initial data $f_0$. Then, for any $f_0\in\sfL^1(m)$,
\begin{equation*}
    \left\Vert f- \left(\int_{\R^d \times \R^d} f_0\right)G\right\Vert_{\sfL^1}\lesssim 
    e^{-\lambda t^\gamma} \left\Vert f_0- \left(\int_{\R^d \times \R^d} f_0 \right) G\right\Vert_{\sfL^1(m)},
\end{equation*}
where 
\begin{equation*}
    \gamma := 
    \begin{cases}
\min \left(1,\frac{\beta}{2} \right), & \text{when } \alpha = 1 \text{ and } \beta >0,\\[5pt]
\frac{\alpha}{2-\alpha} & \text{when } \alpha < 1 \text{ and } \beta \geq 2,\\[5pt]
\frac{\alpha\beta}{2(2-\alpha)} & \text{when } \alpha < 1 \text{ and } \beta < 2.\\
    \end{cases}
\end{equation*}
\end{theorem} 
We believe that the rates for $\beta \geq 2$ and any $\alpha$ are optimal. We improve the result by Cao  \cite{Cao2019}, the sharp value $\frac{2(1-\alpha)}{\alpha}$ being reachable when $\beta = 2$ and $\alpha \leq 1$. Observe that it is claimed here that $G$ is in fact a function and not a measure. This is coming from the fact that the Fokker-Planck operator $\cL$ has regularising properties. 

We believe that approaches of the present paper could be used to handle (part of) the case of a polynomial local equilibrium of the form
\begin{equation*}
    \cM \propto \wv^{-d-\gamma},
\end{equation*}
with $\gamma > 0$. In this case, the velocity corrector $\chi$ is actually miraculously explicit:
\begin{equation*}
 \chi(v):=\frac{(3+|v|^2)v}{d+3\gamma-4}, \qquad\textrm{ so that} \qquad  \nabla_v\chi(v)
 =\frac{(3+|v|^2)\textbf{I}  +2v\otimes v}{d+3\gamma-4},
\end{equation*}
and hence, for every $\xi\in\R^d$,
\begin{equation}\label{eq:Dchi-bounds}
 \frac{\vert v \vert^2}{d+3\gamma-4}|\xi|^2
 \le \xi\cdot\nabla_v\chi(v)\xi
 \le \frac{3\vert v \vert^2}{d+3\gamma-4}|\xi|^2.
\end{equation}
Thus, $\nabla_v\chi$ is positive definite if and only if $d+3\gamma-4 >0$, which gives a restriction on $\gamma$. We leave for a future work the full proofs in this situation.

It is worth mentioning that similar strategies are also interestingly used by the author, Dolbeault and Ziviani \cite{bouindolbeaultziviani2026} to show propagation of weighted energies in factorised kinetic Fokker-Planck equations. 

The remainder of the paper is organised as follows. The following \Cref{sec:corrector} is an independent interlude that describes solutions to an auxiliary problem in velocity that is fundamental to derive Lyapunov functionals later. In \Cref{sec:Lyapunov1}, we separately tackle the cases $\alpha =1$ and $\beta \leq 1$ that do not require the corrector of \Cref{sec:corrector} and show the importance of a cut-off in velocity.  \Cref{sec:Lyapunov2} treats the case $\alpha =1$ and $\beta > 1$. The remaining sections address the more difficult regime $\alpha < 1$. In \Cref{sec:Lyapunov3}, we describe our approach for $\beta >2$, then in \Cref{sec:Lyapunov4} we focus on the case $\beta <2$. We revisit the case $\beta = 2$ in the last \Cref{sec:Lyapunov5}. Except for the velocity corrector introduced in \Cref{sec:corrector}, the notation used in each
Lyapunov construction is local to the corresponding section unless stated otherwise.

\section{Interlude: the velocity cell corrector}\label{sec:corrector}

In this section, we examine the following problem in the velocity variable, 
\begin{equation}\label{eq:Poisson}
 \sfL^\star\chi=-v\qquad\text{in }\R^d.
\end{equation}
supplemented with a normalisation condition, 
\begin{equation}\label{eq:normalization}
 \int_{\R^d}\chi(v)\, \cM(v) \dd v=0.
\end{equation}
The equation is understood component-wise.

\begin{definition}\label{def:weak}
A function $\chi\in H^1(\cM(v) \dd v;\R^d)$ is a weak solution of
\eqref{eq:Poisson} if, for every $i\in\{1,\dots,d\}$ and every
$\varphi\in C_c^\infty(\R^d)$,
\begin{equation}\label{eq:weak}
 \int_{\R^d}\nabla\chi_i\cdot\nabla\varphi\,\cM(v) \dd v
 =\int_{\R^d}v_i \,  \varphi\,\cM(v) \dd v.
\end{equation}
\end{definition}
We now state the main result. The constants implicit in $O(\cdot)$ and
$\asymp$ may depend on $d$ and $\beta$, but never on the radial variable
$r$.
\begin{theorem}[Velocity cell problem : Existence, uniqueness, and asymptotics]\label{thm:cell}

For every $d\ge1$ and $\beta>0$, problem \eqref{eq:Poisson}-\eqref{eq:normalization}
has a unique solution $\chi\in H^1(\cM(v) \dd v;\mathbb R^d)$. This solution is smooth,
odd and $O_d(\R)$-equivariant. Hence
\begin{equation*}
 \chi(v)=p_\beta(|v|)v
       =g_\beta(|v|)\frac{v}{|v|},
 \qquad g_\beta(r):=r p_\beta(r),
\end{equation*}
with the usual smooth interpretation at $v=0$. Moreover,
for constants depending only on $(d,\beta)$,
\begin{equation}\label{eq:cell-growth-short}
 |\chi(v)|\le A_\beta(\wv) := C_{d,\beta}
 \begin{cases}
  \lfloor v\rceil^{3-\beta},&0<\beta<3,\\
  \ln(e+\lfloor v\rceil),&\beta=3,\\
  1,&\beta>3,
 \end{cases}
\end{equation}
and
\begin{equation}\label{eq:cell-gradient-upper-short}
 |D_v\chi(v)|\le
 \begin{cases}
  C_{d,\beta}\lfloor v\rceil^{2-\beta},&0<\beta<2,\\
  1,&\beta\ge2.
 \end{cases}
\end{equation}
Most importantly, $D_v\chi(v)$ is symmetric positive definite and there is a constant $c_\chi$ such that
\begin{equation}\label{eq:cell-coercivity}
 \xi\cdot D_v\chi(v)\xi
 \ge c_{\chi}\lfloor v\rceil^{2-\beta}|\xi|^2,
 \qquad v,\xi\in\mathbb R^d,
\end{equation}
Finally, when $\beta=2$, one has exactly $\chi(v)=v$.
\end{theorem}

\begin{proof}[{\bf Proof of \Cref{thm:cell}}]

We seek a smooth solution of the form $\chi(v):=p_\beta(|v|)v$ satisfying 
$\mathsf L^\star\chi=-v$. Observe that $p_\beta$ should satisfy
\begin{equation}\label{eq:p-ode-short}
 p_\beta''+\left(\frac{d+1}{r}-r \wr^{\beta-2}\right)p_\beta'
 -\wr^{\beta-2}p_\beta=-1,
\end{equation}
if $r$ is the variable $\vert v \vert$. Observe interestingly that writing $d+1 = (d+2)-1$, the latter equation may be seen as the equation satisfied by a radial solution to 
\begin{equation*}
 \int_{\mathbb R^{d+2}}
 \bigl(\nabla P\cdot\nabla\varphi+\wx^{\beta-2} P\varphi\bigr) \cM(x) \dd x
 =\int_{\mathbb R^{d+2}} \varphi \cM(x) \dd x.
\end{equation*}
The left-hand side is coercive and the right-hand side is continuous because
$\wv^{2-\beta} \cM(v)$ is integrable. The Lax--Milgram theorem, rotational
invariance and the maximum principle give a unique smooth positive radial
solution $P(x)=p_\beta(|x|)$. 

For $g_\beta=rp_\beta$ and
$\rho_\beta(r):=r^{d-1}e^{- \frac{\wr^\beta}{\beta}}$, one has, after rewriting under divergence form,
\begin{equation*}
 -\bigl(\rho_\beta g_\beta'\bigr)'
 +(d-1)\frac{\rho_\beta}{r^2}g_\beta=r\rho_\beta.
\end{equation*}
The energy estimate associated with the preceding variational problem gives
$\rho_\beta(r)g_\beta'(r)\to0$. Hence
\begin{equation*}
 g_\beta'(r)=r^{1-d}e^{ \frac{\wr^\beta}{\beta}}\int_r^\infty s^{d}e^{- \frac{\ws^\beta}{\beta}}\,ds
 -(d-1)r^{1-d}e^{ \frac{\wr^\beta}{\beta}}\int_r^\infty
       s^{d-3}e^{- \frac{\ws^\beta}{\beta}}g_\beta(s)\,ds.
\end{equation*}
A one-dimensional Laplace estimate first gives an upper bound on $g_\beta$
and then shows that the second term is lower order. This yields
\eqref{eq:cell-growth-short} and \eqref{eq:cell-gradient-upper-short}.

Observe that
\begin{equation*}
 D_v \chi = \left( \textbf{I} - \frac{v \otimes v}{\vert v \vert^2} \right) \frac{g(\vert v\vert)}{\vert v\vert} + \left( \frac{v \otimes v}{\vert v \vert^2} \right)g'(\vert v\vert)
\end{equation*}
so it is symmetric, and its eigenvalues are explicit, 
\begin{equation*} p_\beta(r)=\frac{g_\beta(r)}r, \qquad \text{ and }
 \qquad g_\beta'(r).
\end{equation*}
Previous estimates happen to give \eqref{eq:cell-coercivity}.
\end{proof}

\section{The case \texorpdfstring{$\alpha =1$}{alpha =1} and \texorpdfstring{$\beta \leq 1$}{betaleq1}}\label{sec:Lyapunov1}

In this section, we prove \Cref{thm:main-lyapunov} in the case $\alpha =1$ and $\beta \in (0,1]$. For this case, we extend the idea of \cite{BouinZiviani}. Let $\psi\in C_c^\infty([0,\infty))$ be such that
\begin{equation*}
  0\le \psi\le 1,
  \qquad
  \psi=1\text{ on }[0,1],
  \qquad
  \psi=0\text{ on }[2,\infty).
\end{equation*}
For $R\ge2$, define
\begin{equation*}
  \psi_R(y):=\psi\!\left(\frac{y}{R}\right),
  \qquad y\ge0.
\end{equation*}
Then $\psi_R(\wv )=1$ for $\wv \le R$, $\psi_R(\wv )=0$ for $\wv \ge2R$, and
\begin{equation*}
  |\psi_R'|\le \frac{C_\psi}{R},
  \qquad
  |\psi_R''|\le \frac{C_\psi}{R^2}.
\end{equation*}
We define, for $\eps$, $A$ and $R$ to be chosen later, 
\begin{equation}\label{eq:HA-def}
  H=E \left( A + E +
  \varepsilon \,\frac{\psi_R(\wv )}{\wx} \, v \cdot x\right),
\end{equation}
and, to ease readability, we may also use the shorthand
\begin{equation*}
    Q_R(x,v) := \psi_R(\wv )\, \frac{x \cdot v}{\wx}.
\end{equation*}
Finally, for $0<\theta\le\frac{\beta}{2}$, define
\begin{equation*}
  m_\theta :=\exp\!\left(\delta H^{\frac{\theta}{2}}\right).
\end{equation*}

The following theorem establishes \Cref{thm:main-lyapunov} in the present parameter regime.
\begin{theorem}\label{thm:lyap-alpha1-betaleq1}
Let $0<\theta\le\frac{\beta}{2}$.  For $\delta>0$ small enough, $\eps$, $A$ and $R$ well chosen, $m_\theta$ satisfies
\begin{equation}\label{eq:mtheta-lyap}
  \mathcal L^{\star} m_\theta
  \le C\mathbf 1_{\mathcal{K}}-c\,\frac{m_\theta}{(\ln(e+ m_\theta))^{\frac{1-\theta}{\theta}}},
\end{equation}
for some compact set $\mathcal{K}$ and some positive constants $c$ and $C$.
\end{theorem}

\subsection{Preliminary computations}

\begin{lemma}\label{lem:exact-formula}
For $H$ defined in \eqref{eq:HA-def},
\begin{align*}\label{eq:LHA-exact}
  \mathcal L^{\star} H 
  &=
  2 \left(d-\wv ^{\beta-2}|v|^2\right) E +\varepsilon \psi_R(\wv )\left\{ \frac{|v|^2-\wx^{-2} ( v \cdot x )^2}{\wx} - \frac{\vert x\vert^2}{\wx^2} - \wv ^{\beta-2}\frac{v\cdot x}{\wx}\right\} E\\
  & +\varepsilon  \psi_R'(\wv )\frac{v \cdot x}{\wx}\left(
            \frac{d+2}{\wv }
            - \left(1+ \wv ^\beta\right)\frac{|v|^2}{\wv ^3}
            - \frac{v \cdot x}{\wx\wv } 
          \right) E +\varepsilon \psi_R''(\wv ) \frac{v \cdot x}{\wx}\frac{|v|^2}{\wv ^2} E\\       
  &+ A\left(d-\wv ^{\beta-2}|v|^2\right)+2|v|^2+\varepsilon\psi_R(\wv )\frac{v \cdot x}{\wx}
    \left(d+2-\wv ^{\beta-2}|v|^2\right) +2\varepsilon\frac{v \cdot x}{\wx} \psi_R'(\wv )\frac{|v|^2}{\wv }.
\end{align*}
\end{lemma}

\begin{proof}[{\bf Proof of \Cref{lem:exact-formula}}]
Recall that in this setting, the formal dual operator is therefore
\begin{equation}\label{eq:dual-setting}
  \mathcal L^{\star} h 
  =v\cdot\nabla_x h-\wx^{-1}x \cdot\nabla_v h
   +\Delta_v h-\wv ^{\beta-2}v\cdot\nabla_v h.
\end{equation}

\noindent {\bf \# Step 1.} First,
\begin{equation*}
  D_x \left( \frac{x}{\wx} \right) =\frac1{\wx}\left(I-\frac{x}{\wx} \otimes\frac{x}{\wx} \right).
\end{equation*}
Since $\psi_R(\wv )$ depends only on $v$,
\begin{equation*}
  v\cdot\nabla_xQ_R
  =\psi_R(\wv )\frac{|v|^2-\wx^{-2} \left(v \cdot x\right)^2}{\wx}.
\end{equation*}
Moreover,
\begin{equation*}
  \nabla_vQ_R
  =\psi_R(\wv )\frac{x}{\wx} 
  +\psi_R'(\wv )\frac{ v \cdot x}{\wx\wv }v,
\end{equation*}
and therefore
\begin{equation*}
  \frac{x}{\wx} \cdot\nabla_vQ_R
  =\psi_R(\wv )\frac{\vert x \vert^2}{\wx^2}
  +\psi_R'(\wv )\frac{\wx^{-2} \left(v \cdot x\right)^2}{\wv }.
\end{equation*}
This proves 
\begin{align}
  &\bigl(v\cdot\nabla_x-\frac{x}{\wx} \cdot\nabla_v\bigr)Q_R
  \nonumber\\
  &\qquad=
    \psi_R(\wv )\left(\frac{|v|^2-\wx^{-2} \left(v \cdot x\right)^2}{\wx}
    -\frac{\vert x \vert^2}{\wx^2}\right)
    -\psi_R'(\wv )\frac{\wx^{-2} \left(v \cdot x\right)^2}{\wv },
    \label{eq:TQ}\end{align}

For the velocity diffusion part, since $\frac{x}{\wx} $ is independent of $v$, it is enough to compute
\begin{equation*}
  \Delta_v\bigl(\psi_R(\wv )v_i\bigr)
  =\left[
    \psi_R''(\wv )\frac{|v|^2}{\wv ^2}
    +\psi_R'(\wv )\left(
        \frac{d+2}{\wv }-\frac{|v|^2}{\wv ^3}
      \right)
   \right]v_i
\end{equation*}
and
\begin{equation*}
  -\wv ^{\beta-2}v\cdot\nabla_v\bigl(\psi_R(\wv )v_i\bigr)
  =-\left(\wv ^{\beta-3}\psi_R'(\wv )|v|^2
  +\wv ^{\beta-2}\psi_R(\wv )\right)v_i.
\end{equation*}
This proves 
\begin{align}
&\bigl(\Delta_v-\wv ^{\beta-2}v\cdot\nabla_v\bigr)Q_R
  \nonumber\\
  &\qquad=( \frac{v \cdot x}{\wx} )\Bigg[
  \psi_R''(\wv )\frac{|v|^2}{\wv ^2}
  +\psi_R'(\wv )\left(
      \frac{d+2}{\wv }
      -\frac{|v|^2}{\wv ^3}
      -\wv ^{\beta-3}|v|^2
    \right)
  -\wv ^{\beta-2}\psi_R(\wv )
  \Bigg].\label{eq:velocityQ}    
\end{align}
Finally,
\begin{align}
  2\nabla_vE\cdot\nabla_vQ_R
  &=2v\cdot\left(
    \psi_R(\wv )\frac{x}{\wx} 
    +\psi_R'(\wv )\frac{ v \cdot x }{\wx\wv }v
  \right)\nonumber\\
  &=2 \left( \frac{v \cdot x}{\wx} \right)
\left(\psi_R(\wv )+\psi_R'(\wv )\frac{|v|^2}{\wv }\right).
  \label{eq:grad-cross}
\end{align}

\noindent{\bf \# Step 2.}  The transport part annihilates the energy:
\begin{equation*}
\sfT E =  \bigl(v\cdot\nabla_x-\frac{x}{\wx} \cdot\nabla_v\bigr)E=0,
\end{equation*}
and the velocity part gives
\begin{equation*}
  \bigl(\Delta_v-\wv ^{\beta-2}v\cdot\nabla_v\bigr)E
  =d-\wv ^{\beta-2}|v|^2.
\end{equation*}
Thus, the energy part gives
\begin{equation*}
  \mathcal L^{\star}(E^2+AE)
  =(2E+A)\left(d-\wv ^{\beta-2}|v|^2\right)+2|v|^2.
\end{equation*}
Since the transport part annihilates $E$,
\begin{equation*}
\begin{aligned}
  \mathcal L^{\star}(EQ_R)
  ={}&E\bigl(v\cdot\nabla_x-\frac{x}{\wx} \cdot\nabla_v\bigr)Q_R
  +Q_R\left(d-\wv ^{\beta-2}|v|^2\right) \\
  &+E\bigl(\Delta_v-\wv ^{\beta-2}v\cdot\nabla_v\bigr)Q_R
  +2\nabla_vE\cdot\nabla_vQ_R.
\end{aligned}
\end{equation*}
Substituting the identities obtained in {\bf \#Step 1} yields the asserted formula.
\end{proof}

\subsection{The Lyapunov estimate for $H$}

\begin{proposition}\label{prop:estH}
Choose $\varepsilon$ sufficiently large (as explicitly given from the proof - see \eqref{eq:eps-star}). Then, one can choose, in this order, $R\gg1$, $A\gg1$, and $X_0\gg1$ such that $H$ defined by \eqref{eq:HA-def} satisfies
\begin{equation}\label{eq:equivalence}
H \asymp E^2+AE,
\end{equation}
and
\begin{equation}\label{eq:HA-lyap}
  \mathcal L^{\star} H
  \le C\mathbf 1_K-cE,
  \qquad
  K:=\{(x,v):\ \wx\le X_0,\ \wv \le 2R\}.
\end{equation}
In fact, after increasing $R$ if needed, the high-velocity region satisfies the stronger estimate
\begin{equation}\label{eq:high-vel-strong}
  \mathcal L^{\star} H\le -cE\wv ^\beta
  \qquad\text{on }\{\wv \ge2R\}.
\end{equation}
\end{proposition}

\begin{proof}[{\bf Proof of \Cref{prop:estH}}]

Start with a technical definition. For $0<\beta\le1$ and $u\ge0$,
\begin{equation*}
  \wu^{\beta-2}u<1,
  \qquad
  \wu^{\beta-2}u^2\to+\infty
  \quad\text{as }u\to+\infty.
\end{equation*}
Define
\begin{equation}\label{eq:eps-star}
  \varepsilon_*
  :=\sup_{u\ge0}
  \frac{\bigl(2d-2\wu^{\beta-2}u^2\bigr)_+}
       {1-\wu^{\beta-2}u}.
\end{equation}
This number is finite. Choose $\varepsilon>\varepsilon_*$. Then there exists $\kappa_0>0$ such that
\begin{equation}\label{eq:key-one-dimensional}
  2d-2\wu^{\beta-2}u^2
  +\varepsilon\bigl(\wu^{\beta-2}u-1\bigr)
  \le -4\kappa_0
  \qquad\text{for all }u\ge0.
\end{equation}

We shall now prove the norm equivalence. It remains only to prove the equivalence \eqref{eq:equivalence}.  Since
\begin{equation*}
  |Q_R|\le |v|\mathbf 1_{\wv \le2R}\le2R,
\end{equation*}
choosing
\begin{equation*}
  A\ge4\varepsilon R
\end{equation*}
gives
\begin{equation*}
  E+\frac{A}{2} \leq E+A-2\varepsilon R \leq E+A+\varepsilon Q_R\le E+A+2\varepsilon R\le E+\frac{3A}{2}.
\end{equation*}
Because $H=E(E+A+\varepsilon Q_R)$, this proves \eqref{eq:equivalence}.

To show the Lyapunov estimate, we split the proof into the three velocity regions.  The order of the parameter choices is important: $A$ is fixed before the final choice of $X_0$, because some bounded remainders are proportional to $A$.

\medskip
\noindent\textbf{\# Zone 1. The low-velocity region $\wv \le R$.}
Here $\psi_R(\wv )=1$ and $\psi_R'(\wv )=\psi_R''(\wv )=0$. Hence \Cref{lem:exact-formula} becomes
\begin{align}\label{eq:low-zone}
  \mathcal L^{\star} H
  ={}&
  E\left[
      2d-2\wv ^{\beta-2}|v|^2
      +\varepsilon\left(
        \frac{|v|^2-\wx^{-2} \left(v \cdot x\right)^2}{\wx}
        -\frac{\vert x \vert^2}{\wx^2}
        -\wv ^{\beta-2} \frac{v \cdot x}{\wx} 
      \right)
    \right]
  \nonumber\\
  &+A\left(d-\wv ^{\beta-2}|v|^2\right)
   +2|v|^2
   +\varepsilon \left( \frac{v \cdot x}{\wx} \right)\left(d-\wv ^{\beta-2}|v|^2\right)
   +2\varepsilon \frac{v \cdot x}{\wx}  .
\end{align}
The last line is bounded from above by a constant depending on $A,\varepsilon,R,d$.  It remains to control the coefficient of $E$. Using
\begin{equation*}
  -\wv ^{\beta-2} \frac{v \cdot x}{\wx} 
  \le \wv ^{\beta-2}\frac{\vert x \vert}{\wx} \,|v|,
\end{equation*}
and, for $0\le \eta\le1$ and $a\ge0$,
\begin{equation*}
  -\eta^2+a\eta\le a-1+2(1-\eta),
\end{equation*}
with $\eta=\frac{\vert x \vert}{\wx}$ and $a=\wv ^{\beta-2}|v|$, we get
\begin{align*}
  &2d-2\wv ^{\beta-2}|v|^2
  +\varepsilon\left(
        \frac{|v|^2-\wx^{-2} \left(v \cdot x\right)^2}{\wx}
        -\frac{\vert x \vert}{\wx}
        -\wv ^{\beta-2} \frac{v \cdot x}{\wx} 
      \right)
  \\
  &\quad\le
  2d-2\wv ^{\beta-2}|v|^2
  +\varepsilon\bigl(\wv ^{\beta-2}|v|-1\bigr)
  +\varepsilon\left(\frac{|v|^2}{\wx}+2\left(1-\frac{\vert x \vert}{\wx}\right)\right).
\end{align*}
Since $|v|\le R$ in this region and $1-\frac{\vert x \vert}{\wx}\le\wx^{-2}$, \eqref{eq:key-one-dimensional} implies, after choosing $X_0$ large enough, that for $\wx\ge X_0$,
\begin{equation*}
  2d-2\wv ^{\beta-2}|v|^2
  +\varepsilon\left(
        \frac{|v|^2-\wx^{-2} (v \cdot x)^2}{\wx}
        -\frac{|x|^2}{\wx} 
        -\wv ^{\beta-2} \frac{v \cdot x}{\wx} 
      \right)
  \le -2\kappa_0.
\end{equation*}
Therefore, on $\{\wv \le R,\ \wx\ge X_0\}$,
\begin{equation}\label{eq:low-final}
  \mathcal L^{\star} H\le C_{A,\varepsilon,R}-2\kappa_0E.
\end{equation}
After increasing $X_0$ once more, the bounded remainder is absorbed by the negative multiple of $E$ outside the compact.

\medskip
\noindent\textbf{\# Zone 2. The transition region $R\le\wv \le2R$.}
The derivatives of the cut-off are supported here.  From the bounds on $\psi_R'$ and $\psi_R''$, and using $0<\beta\le1$, one has
\begin{align}\label{eq:velocityQ-bound}
  &\left|
  \psi_R''(\wv )\frac{|v|^2}{\wv ^2}
  +\psi_R'(\wv )\left(
      \frac{d+2}{\wv }
      -\frac{|v|^2}{\wv ^3}
      -\wv ^{\beta-3}|v|^2
    \right)
  -\wv ^{\beta-2}\psi_R(\wv )
  \right|
  \le \frac{C_{\psi,d}}{R}.
\end{align}
Indeed,
\begin{equation*}
  \left|\psi_R''(\wv )\frac{|v|^2}{\wv ^2}\right|\lesssim R^{-2},
\end{equation*}
\begin{equation*}
  |\psi_R'(\wv )|\left(
      \frac{d+2}{\wv }+\frac{|v|^2}{\wv ^3}+\wv ^{\beta-3}|v|^2
    \right)\lesssim R^{-1},
  \qquad
  \wv ^{\beta-2}\lesssim R^{-1}.
\end{equation*}
Consequently,
\begin{align*}
 & \left| \frac{v \cdot x}{\wx} \right| \left|
  \psi_R''(\wv )\frac{|v|^2}{\wv ^2}
  +\psi_R'(\wv )\left(
      \frac{d+2}{\wv }
      -\frac{|v|^2}{\wv ^3}
      -\wv ^{\beta-3}|v|^2
    \right)
  -\wv ^{\beta-2}\psi_R(\wv )
  \right|
  \le C_{\psi,d},\\
 & \left|\psi_R'(\wv )\frac{\left( v \cdot x \right)^2}{\wv \wx^2}\right|
  \le C_\psi.
\end{align*}
On the other hand,
\begin{equation*}
  \wv ^{\beta-2}|v|^2
  =\wv ^\beta-\wv ^{\beta-2}
  \ge cR^\beta
\end{equation*}
for $R$ sufficiently large. Thus the coefficient of $E$ in \Cref{lem:exact-formula} is bounded above by
\begin{equation*}
  2d-cR^\beta+\frac{4\varepsilon R^2}{\wx}+C_{\psi,d}\varepsilon.
\end{equation*}
Choose $R$ sufficiently large, after $\varepsilon$ has been fixed, so that
\begin{equation*}
  2d-cR^\beta+C_{\psi,d}\varepsilon\le -4\kappa_1
\end{equation*}
for some $\kappa_1>0$.  After $A$ has been fixed, choose $X_0$ larger, if necessary, so that
\begin{equation*}
  \frac{4\varepsilon R^2}{\wx}\le \kappa_1
  \qquad\text{for }\wx\ge X_0,
\end{equation*}
and so that the bounded remainders are absorbed.  The terms in \Cref{lem:exact-formula} not multiplied by $E$ are bounded above in the transition annulus, with a bound depending on $A,\varepsilon,R,d$. Hence, on $\{R\le\wv \le2R,\ \wx\ge X_0\}$,
\begin{equation}\label{eq:transition-final}
  \mathcal L^{\star} H\le C_{A,\varepsilon,R}-\kappa_1E,
\end{equation}
and the final enlargement of $X_0$ gives the desired exterior estimate.

\medskip
\noindent\textbf{\# Zone 3. The high-velocity region $\wv \ge2R$.}
Here $\psi_R(\wv )=0$, so the corrector vanishes and
\begin{equation*}
  \mathcal L^{\star} H=(2E+A)\left(d-\wv ^{\beta-2}|v|^2\right)+2|v|^2 \leq (2E+A)\left(d-\wv ^{\beta-2}|v|^2\right) + 4E,
\end{equation*}
since $E\ge \frac{|v|^2}{2}$. For $R$ sufficiently large,
\begin{equation*}
  d-\wv ^{\beta-2}|v|^2\le -\frac34\wv ^\beta.
\end{equation*}
 and $\wv ^\beta\to\infty$, increasing $R$ once more gives
\begin{equation*}
  \mathcal L^{\star} H\le -\frac12E\wv ^\beta,
\end{equation*}
which is \eqref{eq:high-vel-strong}.

Combining \eqref{eq:low-final}, \eqref{eq:transition-final}, and \eqref{eq:high-vel-strong} gives \eqref{eq:HA-lyap}, with
$K=\{\wx\le X_0,\wv \le2R\}$.

\end{proof}

\subsection{Taking the exponential }

\begin{proof}[{\bf Proof of \Cref{thm:lyap-alpha1-betaleq1}}]
Let $\Phi(H)=\exp(\delta H^{\frac{\theta}{2}})$. Since $0<\theta\le1$, one has
\begin{equation*}
  \Phi''(H)\le \frac{\delta\theta}{2}H^{\frac{\theta}{2}-1}\Phi'(H).
\end{equation*}
Therefore
\begin{equation}\label{eq:chain}
  \mathcal L^{\star} m_\theta
  \le \Phi'(H)\left(
    \mathcal L^{\star} H+\frac{\delta\theta}{2}H^{\frac{\theta}{2}-1}|\nabla_vH|^2
  \right).
\end{equation}
For the above choice of parameters, there is a constant $C>0$ such that
\begin{equation}\label{eq:grad-bound}
  |\nabla_vH|^2\le C E^2(1+\wv ^2)
  \qquad\text{for all }(x,v).
\end{equation}
Moreover, on $\{\wv \le2R\}$ one has the sharper bound
\begin{equation}\label{eq:grad-low-bound}
  |\nabla_vH|^2\le C_R E^2.
\end{equation}
Indeed, differentiating $H$ in $v$ gives
\begin{equation*}
  \nabla_vH=(2E+A+\varepsilon Q_R)v+\varepsilon E\nabla_vQ_R.
\end{equation*}
The quantity $Q_R$ is bounded by $2R$ on its support. Also
\begin{equation*}
  \nabla_vQ_R
  =\psi_R(\wv )\frac{x}{\wx} 
  +\psi_R'(\wv )\frac{ v \cdot x }{\wx\wv }v
\end{equation*}
is bounded uniformly for fixed $R$. This proves \eqref{eq:grad-low-bound} on $\{\wv \le2R\}$.  Outside this region, the corrector vanishes and
\begin{equation*}
  \nabla_vH=(2E+A)v,
\end{equation*}
which is bounded by $CE\wv $. This proves \eqref{eq:grad-bound}.

In the bounded-velocity exterior region $\{\wv \le2R,\ \wx\ge X_0\}$, \Cref{prop:estH} and \eqref{eq:grad-bound} give
\begin{equation*}
  \mathcal L^{\star} H+\frac{\delta\theta}{2}H^{\frac{\theta}{2}-1}|\nabla_vH|^2
  \le -cE+C\delta E^\theta.
\end{equation*}
Since $\theta\le1$, enlarging $X_0$ if necessary and then choosing $\delta$ small makes this bounded above by $-cE/2$ outside $K$.

In the high-velocity region $\{\wv \ge2R\}$, \Cref{prop:estH} gives $\mathcal L^{\star} H\le-cE\wv ^\beta$, while \eqref{eq:grad-bound} gives
\begin{equation*}
  H^{\frac{\theta}{2}-1}|\nabla_vH|^2
  \lesssim E^{\theta-2}E^2\wv ^2
  =E^\theta\wv ^2.
\end{equation*}
Since $E\ge \frac14 \wv ^2$ for $\wv $ large and $\theta\le\frac{\beta}{2}$,
\begin{equation*}
  E^\theta\wv ^2
  =E\wv ^\beta\,E^{\theta-1}\wv ^{2-\beta}
  \lesssim E\wv ^\beta\,\wv ^{2\theta-\beta}
  \lesssim E\wv ^\beta.
\end{equation*}
Choosing $\delta$ small enough absorbs the second-order term in \eqref{eq:chain}. Hence
\begin{equation*}
  \mathcal L^{\star} m_\theta
  \le C\mathbf 1_K-c\Phi'(H)E.
\end{equation*}
Finally, by \eqref{eq:equivalence}, $H\asymp E^2$ outside a compact set, and hence
\begin{equation*}
  \Phi'(H)E
  =\frac{\delta\theta}{2}H^{\frac{\theta}{2}-1}E\,m_\theta
  \asymp H^{-(1-\theta)/2}m_\theta.
\end{equation*}
Since
\begin{equation*}
  \ln m_\theta=\delta H^{\frac{\theta}{2}},
\end{equation*}
we have
\begin{equation*}
  H^{-(1-\theta)/2}
  \asymp\delta^{\frac{1-\theta}{\theta}}(\ln(e+ m_\theta))^{-\frac{1-\theta}{\theta}}.
\end{equation*}
This proves \eqref{eq:mtheta-lyap}.
\end{proof}

\section{The case \texorpdfstring{$\alpha =1$}{} and \texorpdfstring{$\beta > 1$}{}}\label{sec:Lyapunov2}

Throughout this section,
\begin{equation*}
0 < \theta \leq \min\left(1,\frac\beta2\right).
\end{equation*}
Fix $R>1$ and let $\eta_R\in C^\infty(\R^d;[0,1])$ be radial, with
\begin{equation*}
  \eta_R(v)=1\quad \text{when } |v|\le R,
  \qquad
  \eta_R(v)=0\quad \text{when } |v|\ge 2R,
\end{equation*}
and
\begin{equation*}
  |\nabla_v\eta_R|\lesssim R^{-1},
  \qquad
  |\nabla_v^2\eta_R|\lesssim R^{-2}.
\end{equation*}
Define
\begin{equation}\label{eq:F-R}
F_R(x,v):=\wx^\theta+\theta\wx^{\theta-2} \eta_R(v) \,  x\cdot\chi(v).
\end{equation}
For $\kappa>0$ and $K>0$ to be chosen, put
\begin{equation}\label{eq:M0-W-m}
  M_0:=\exp\bigl(\kappa F_R\bigr),
  \qquad
  W:=\exp\bigl(\kappa E^\theta\bigr),
\end{equation}
and
\begin{equation}\label{eq:m-def}
  m:=M_0+K W.
\end{equation}
The additive
construction $M_0+KW$ is useful here because the linearity of
$\mathcal L^\star$ avoids any additional cross term between the two
weights.

First note that, since $\eta_R\chi$ is bounded and $\theta \leq 1$,
\begin{equation}\label{eq:M0-comparison}
M_0(x,v) \asymp e^{\kappa \wx^\theta}.
\end{equation}
Thus $m$ is positive and coercive. The following theorem establishes \Cref{thm:main-lyapunov} in the present parameter regime. 
\begin{theorem}\label{thm:alpha1-betagt1}
There exist constants $R>1$, $K>0$, $\kappa>0$, $c>0$, $C>0$, and a compact set $\cK\subset\R^d\times\R^d$ such that the weight $m$ defined by \eqref{eq:m-def} satisfies
\begin{equation}\label{eq:scalar-drift}
  \cLs m
  \le C\one_{\cK}
  -c \kappa^{\frac{1}{\theta}}\frac{m}{\left( \ln(e+m)\right)^{\frac{1-\theta}{\theta}}}.
\end{equation}
When $\theta=1$, the denominator in \eqref{eq:scalar-drift} is interpreted as $1$.   
\end{theorem}


Note that when $\beta\ge2$, one may take $\theta=1$ in \Cref{thm:alpha1-betagt1} and the estimate \eqref{eq:scalar-drift} reduces to the geometric Lyapunov condition
\begin{equation*}
  \cLs m\le C\one_{\cK}-c \kappa m.
\end{equation*}
Thus, in the range $\beta\ge2$ no sub-geometric logarithmic loss is produced by this Lyapunov computation.

\subsection{Estimate on $M_0$}

\begin{lemma}[Estimate on $M_0$]\label{lem:M0}
Fix $1\le\rho<R$.  There exist $c_\rho,c_R,C_R,C>0$, a number
$\kappa_R>0$, and a compact set $\cK_R$ such that, for
$0<\kappa\le\kappa_R$,
\begin{align}
  \cLs M_0
  \le{}& C_R\one_{\cK_R}
   -c_\rho\kappa M_0V^{\theta-1}\one_{\{|v|\le\rho\}}
   -c_R\kappa M_0V^{\theta-1}
       \one_{\{\rho<|v|\le R\}} \notag\\
  &+C\kappa R^{\beta-2}A_\beta(R) M_0V^{\theta-1}
       \one_{\{R<|v|<2R\}}
   +C\kappa M_0V^{\theta-1}\wv 
       \one_{\{|v|\ge2R\}}.
  \label{eq:M0-estimate}
\end{align}
Here $c_\rho$ depends on $\rho$ but not on $R$.

\end{lemma}

\begin{proof}[{\bf Proof of \Cref{lem:M0}}]
Since $M_0=e^{\kappa F_R}$,
\begin{equation*}
  \frac{\cLs M_0}{M_0}
  =\kappa\left( \cLs  F_R+\kappa|\nabla_vF_R|^2\right).
\end{equation*}
{\bf \# On $\{|v|\le \rho\}$ and $\{\rho \leq |v|\le R\}$}, $\eta_R=1$ and the cell equation gives the exact cancellation
\begin{equation*}
  v\cdot\nabla_x(\wx^\theta)+\theta\wx^{\theta-2}x \cdot\sfL^\star\chi
  =v\cdot \theta\wx^{\theta-2}x -\theta\wx^{\theta-2}x \cdot v=0.
\end{equation*}
Therefore, on $\{|v|\le R\}$,
\begin{align}\label{eq:LstarFR}
  \cLs F_R
  = \sfT F_R + \sfL^\star F_R
  = v^\top D_x^2(\wx^\theta)\,\chi
  -\theta\wx^{\theta-3} \, x^\top(D_v\chi)x .    
\end{align}
Using \Cref{thm:cell}, the second term satisfies
\begin{equation*}
-\theta\wx^{\theta-3} \, x^\top(D_v\chi)x \le -c_\rho \wx^{\theta-3}|x|^2.
\end{equation*}
on $\{|v|\le \rho\}$ and 
\begin{equation*}
-\theta\wx^{\theta-3} \, x^\top(D_v\chi)x \le -c_R \wx^{\theta-3}|x|^2
\end{equation*}
on $\{\rho \leq |v|\le R\}$. Since
\begin{equation*}
  |D_x^2(\wx^\theta)|
  \lesssim \wx^{\theta-2},
\end{equation*}
the first term in the r.h.s. of \eqref{eq:LstarFR} is absorbed by the second one outside a compact set in $x$.  The term
\begin{equation*}
  |\nabla_vF_R|^2\le A_{\beta}(R)^2\wx^{2\theta-2}
\end{equation*}
and thus is absorbed by taking $\kappa$ small.  This proves the negative part of \eqref{eq:M0-estimate}.

\medskip

\noindent{\bf \# On $R\le |v|\le2R$},
\begin{align*}
  \cLs F_R
  &= \sfT \left(\wx^\theta+\theta\wx^{\theta-2} \eta_R(v) \,  x\cdot\chi(v)\right)+ \sfL^\star (\wx^\theta+\theta\wx^{\theta-2} \eta_R(v) \,  x\cdot\chi(v))\notag\\
  &= \sfT \left(\wx^\theta\right) +\theta \sfT \left(\wx^{\theta-2}\,  x\cdot \eta_R\chi\right) + \theta\wx^{\theta-2} x\cdot\sfL^\star ( \eta_R\,  \chi)\notag\\
  &= \theta\wx^{\theta-2} x\cdot v +\theta \sfT \left(\wx^{\theta-2}\,  x\cdot \eta_R\chi\right) + \theta\wx^{\theta-2} x\cdot\sfL^\star ( \eta_R\,  \chi)\\
  &= \theta \sfT \left(\wx^{\theta-2}\,  x\cdot \eta_R\chi\right) + \theta\wx^{\theta-2} x\cdot\sfL^\star ( (\eta_R-1)\,  \chi).
\end{align*}
where we have used $\sfL^\star (\chi) = -v$. Using the product identities, 
\begin{align*}
&\sfT \left(\wx^{\theta-2}\,  x\cdot \eta_R\chi\right) = \Theta^{-1}v^\top D_x^2 \left(\wx^\theta \right)\eta_R\chi - \wx^{\theta-3} x^\top D_v(\eta_R \chi)\,  x,\\
&\sfL^\star ( (\eta_R-1)\,  \chi) = (\eta_R-1)\, \sfL^\star (  \chi) + \sfL^\star ( (\eta_R-1))\,  \chi + 2 \nabla_v \eta_R \cdot \nabla_v \chi. 
\end{align*}
we deduce that
\begin{align*}
\left\vert\sfT \left(\wx^{\theta-2}\,  x\cdot \eta_R\chi\right) \right\vert \lesssim_R \wv  \wx^{\theta-1},\qquad \left\vert\sfL^\star ( (\eta_R-1)\,  \chi) \right\vert \lesssim_R \wv ,
\end{align*}
which yields the last term in \eqref{eq:M0-estimate}.  

\medskip

\noindent{\bf \# For $|v|\ge2R$}, $\eta_R=0$, so $F_R=\wx^\theta$ and
\begin{equation*}
  \cLs F_R= \theta\wx^{\theta-2}v\cdot x 
  \qquad \textrm{and } \qquad
  |\cLs F_R|\le \theta \wx^{\theta-1}\wv .
\end{equation*}
Combining these estimates and enlarging the compact set gives \eqref{eq:M0-estimate}.
\end{proof}

\subsection{Estimate on the energy weight}

The dissipative radius in the energy estimate is fixed before the cutoff
radius $R$ is chosen.

\begin{lemma}[Estimate on $W$]\label{lem:W}
There exist a radius $\rho\ge2$, a number $\kappa_0>0$, and constants
$c_W,C_W>0$ such that, for every $0<\kappa\le\kappa_0$,
\begin{equation}\label{eq:W-estimate}
  \cLs W
  \le C_W\kappa WV^{\theta-1}\one_{\{|v|\le\rho\}}
  -c_W\kappa WE^{\theta-1}\wv ^\beta
       \one_{\{|v|\ge\rho\}}.
\end{equation}
The constants and the radius $\rho$ are independent of $R$.
\end{lemma}

\begin{proof}[{\bf Proof of \Cref{lem:W}}]
Since $\sfT E=0$, $\cLs W=\sfL^\star W$. Moreover,
\begin{equation*}
  \sfL^\star E=d-\wv ^{\beta-2}|v|^2,
  \qquad
  |\nabla_vE|^2=|v|^2.
\end{equation*}
Thus
\begin{align*}
  \frac{\cLs W}{W}
  &=\kappa\theta E^{\theta-1}\bigl(d-\wv ^{\beta-2}|v|^2\bigr)
+\kappa\theta(\theta-1)E^{\theta-2}|v|^2+\kappa^2\theta^2E^{2\theta-2}|v|^2.
\end{align*}
The middle term is non-positive. Choose $\rho$ large enough that, for
$|v|\ge\rho$,
\begin{equation*}
  \wv ^{\beta-2}|v|^2-d\ge c\wv ^\beta.
\end{equation*}
The ratio of the quadratic term to the principal
negative term is bounded by
\begin{equation*}
\frac{\kappa^2\theta^2E^{2\theta-2}|v|^2}{\kappa\theta E^{\theta-1}\wv ^{\beta-2}|v|^2} = \theta\kappa E^{\theta-1}\wv ^{2-\beta}
  \le \theta\kappa\wv ^{2\theta-\beta}
  \le \theta\kappa,
\end{equation*}
where we used $E\ge \frac{|v|^2}{2}$ and $2\theta\le\beta$.  Taking
$\kappa_0$ small proves the negative part of \eqref{eq:W-estimate}.  On
$|v|\le\rho$, the formula for $\frac{\cLs W}{W}$ and $\theta\le1$ give
\begin{equation*}
  \cLs W\le C_W\kappa WV^{\theta-1}.
\end{equation*}
\end{proof}

\subsection{Proof of \Cref{thm:alpha1-betagt1}}

Starting from \Cref{lem:W} and \Cref{lem:M0}, take $\rho \geq 2$, $R > \rho$ and $\kappa$ so small that
\begin{align*}
  \cLs m
  &\le C_R\one_{\cK_R}
   -c_\rho\kappa M_0V^{\theta-1}\one_{\{|v|\le\rho\}}
   -c_R\kappa M_0V^{\theta-1}
       \one_{\{\rho<|v|\le R\}} 
+C\kappa A_\beta(R) M_0V^{\theta-1}
       \one_{\{R<|v|<2R\}} \\
  & +C\kappa M_0V^{\theta-1}\wv 
       \one_{\{|v|\ge2R\}}
  +KC_W\kappa WV^{\theta-1}\one_{\{|v|\le\rho\}}
  -Kc_W\kappa WE^{\theta-1}\wv ^\beta
       \one_{\{|v|\ge\rho\}}\\
       &\le C_R\one_{\cK_R}
   +\kappa\left( KC_W W-c_\rho M_0\right)V^{\theta-1}\one_{\{|v|\le\rho\}}\\
   &-\left(c_R\kappa M_0V^{\theta-1}+Kc_W\kappa WE^{\theta-1}\wv ^\beta\right)
       \one_{\{\rho<|v|\le R\}}\\       
&+\left(C\kappa R^{\beta-2}A_\beta(R) M_0V^{\theta-1}-Kc_W\kappa WE^{\theta-1}\wv ^\beta\right)
       \one_{\{R<|v|<2R\}} \\
  & +\left(C\kappa M_0V^{\theta-1}\wv -Kc_W\kappa WE^{\theta-1}\wv ^\beta\right)
       \one_{\{|v|\ge2R\}}.
\end{align*}
We first choose $R$ and $K$ so that the inner and exterior absorptions are compatible.  

On $|v|\le \rho$, choose $K$ so small (uniformly in $\kappa \leq 1$) that $KC_W W \leq \frac12 c_\rho M_0$. This is possible since
\begin{equation*}
  \frac{W}{M_0}
  \le C_\rho \exp\{\kappa(E^\theta-\wx^\theta)\} .
\end{equation*}
Moreover, 
\begin{align*}
E^\theta-\wx^\theta &= \left( \frac{\vert v \vert^2}{2} + \wx\right)^\theta-\wx^\theta \leq \left( \frac{\rho^2}{2} + \wx\right)^\theta-\wx^\theta \\
&= \wx^\theta \left(\left( \frac{\rho^2}{2\wx} + 1\right)^\theta-1\right) \leq \theta \wx^{\theta-1} \frac{\rho^2}{2} \leq \theta \frac{\rho^2}{2},
\end{align*}
 since $\theta \leq 1$.
On the annulus $\rho \leq |v|\le R$, the contribution is negative, so there is nothing to do. On $|v|\ge R$ we have, by \eqref{eq:M0-comparison},
\begin{equation*}
  \frac{M_0}{W}
  \le C_R\exp\{-\kappa(E^\theta-\wx^\theta)\} \leq C_R.
\end{equation*}
If $|v|^2\le \wx$, then $E\asymp \wx$ and the exponential factor is bounded by a constant.  Hence
\begin{equation*}
  \frac{M_0\wx^{\theta-1}\wv }{WE^{\theta-1}\wv ^\beta}
  \le C_R\wv ^{1-\beta}
  \le C_RR^{1-\beta}.
\end{equation*}
Since $\beta>1$, this is made smaller than $Kc/4$ by taking $R$ large.  If $|v|^2\ge \wx$, then $E^\theta-\wx^\theta\ge c\min(E^\theta, E^{\theta-1}|v|^2)$ and the exponential factor gives an even stronger bound. 

In the annulus $R < \vert v \vert < 2R$, one sees that
\begin{equation*}
    \frac{R^{\beta-2}A_\beta(R) M_0\wx^{\theta-1}}{ WE^{\theta-1}\wv ^\beta} \lesssim \frac{R^{\beta-2}A_\beta(R)}{ R^\beta} 
\end{equation*}
which is small when $R$ is large. 
Enlarging the compact set covers the bounded remaining region.

Let us tell explicitely that the choice of parameters is in the follwoing order : $\rho$ then $K$ then $R$ then $\kappa$, and the compact $\cK$ at end. 

We have established that there exist positive constants $\rho,R,K,\kappa,c,C$ and a compact set such that $m$ satisfies
\begin{align}
  \cLs m
  \le{}& C\one_{\cK}
  -c\kappa M_0V^{\theta-1}\one_{\{|v|\le R\}} -cK\kappa WE^{\theta-1}\wv ^\beta
       \one_{\{|v|\ge R\}}.
\end{align}
It remains to pass to the scalar form.  On $|v|\le R$, $m\asymp M_0\asymp e^{\kappa \wx^\theta}$, and hence
\begin{equation*}
  \wx^{\theta-1}
  \asymp \kappa^{\frac{1-\theta}{\theta}}
  \left(\ln(e+m)\right)^{-\frac{1-\theta}{\theta}}.
\end{equation*}
On $|v|\ge R$, $m = M_0 + KW\le (C_R + K)W$ and
\begin{equation*}
  E^{\theta-1}\wv ^\beta
  \ge c_R E^{\theta-1}
  \asymp \kappa^{\frac{1-\theta}{\theta}}
  \left(\ln(e+m)\right)^{-\frac{1-\theta}{\theta}}.
\end{equation*}
The proof is now complete. 

\section{The case \texorpdfstring{$\alpha < 1$}{} and \texorpdfstring{$\beta > 2$}{}}\label{sec:Lyapunov3}

We define an auxiliary function $\tau(x) = \wx^{-2(1-\alpha)} = \alpha^{-\frac{2(1-\alpha)}{\alpha}}
   V(x)^{-\frac{2(1-\alpha)}{\alpha}}$. Direct computations give, 
\begin{align*}
& |\nabla V(x)|^2
 =\tau(x)\bigl(1-\wx^{-2}\bigr)\le \tau(x),
\qquad |\nabla V(x)|^2 \ge \frac12\tau(x)\qquad\text{when }\wx\ge\sqrt2,
\\
&\|D_x^2V(x)\| \le \wx^{\alpha-2} =\wx^{-\alpha}\tau(x),
  \qquad 
 |\nabla_x\tau(x)| \le 2(1-\alpha)\frac{\tau(x)}{\wx}.
\end{align*}
Let $h\in C^\infty([0,\infty))$ satisfy
\begin{equation}\label{eq:cutoff-h}
0\le h\le1,
\qquad
h(s)=0\ (s\le1),
\qquad
h(s)=1\ (s\ge2),
\qquad
h'(s)\ge0.
\end{equation}
Define, for the entire section, the following parameters, fixed in
the following order:
\begin{equation}\label{eq:explicit-parameters}
 \begin{gathered}
 \nu:=\frac{\beta+2}{4\beta},
 \qquad
 R_0:=\max\left\{\sqrt8,(8d)^{\frac{1}{\beta}},128^{\frac{1}{\beta}}, \left(16(1 +\Vert h' \Vert_{\infty})\right)^{\frac{1}{\beta-2}},\left(16(1 +\Vert h'' \Vert_{\infty})\right)^{\frac{1}{\beta}}\right\},
 \\
 \eps:=\min\left\{
 \frac12,
 \left(\frac{c_\chi}{8(d+2^{\beta-2})}\right)^{\beta-1},
 \bigl(128(d+2^{\beta-2})\bigr)^{-\frac{\beta-1}{\beta}},c_\chi^{-1}
 \right\},
 \\
 \Phi_0:=1+\frac{2C_\chi^2}{\eps},
 \qquad
 \kappa:=\min\left\{ \frac{c_\chi}{512 \Vert D_v\chi\Vert_\infty^2}
 \eps^{\frac{\beta-2}{\beta-1}} , \frac{1}{512(1+\Vert h' \Vert_\infty) \eps} \right\}.
 \end{gathered}
\end{equation}
Notice that
\begin{equation}\label{eq:nu-properties}
 \frac1\beta<\nu<\frac12,
 \qquad
 1-2\nu(1-\alpha)(3-\beta)_+>0.
\end{equation}

For $0<s\le1$ set
\begin{equation*}
R(s):=R_0s^{-\nu} (\geq R_0),
\end{equation*}
where $R_0\ge1$ will be chosen large.  For $0<s\le1$ define
\begin{equation}\label{eq:eta-q-def}
\eta(s,v):=s+(1-s)h\!\left(\frac{\wv }{R(s)}\right),
\qquad
q(x,v):=\eta(\tau(x),v)\frac{|v|^2}{2}.
\end{equation}
To ease readability, we introduce the following shorthands (slightly abusing notation in the indices).
\begin{equation*}
\eta_{\tau}(v) := \eta(\tau,v).
\end{equation*}
The coefficient $\eta_\tau$ is of size $|\nabla V|^2$ at moderate velocities and becomes equal to $1$ at very large velocities. Set
\begin{equation*}
F(x,v):=V(x)+\nabla_x V(x)\cdot\chi(v).
\end{equation*}
Given $\eps$ and $\Phi_0$ to be chosen later, we define the Lyapunov phase $\Phi$ and the Lyapunov weight $m$ as
\begin{equation*}
\Phi(x,v):=\Phi_0+ F+\eps q, \qquad \textrm{and} \qquad m:=\exp(\kappa\Phi).
\end{equation*}

Define the following explicit compact $\mathcal K$:
\begin{align}
 \mathcal K:= {}&
 \left\{(x,v):
 (1-\alpha)2^{(3-\beta)_+}R_0^{(3-\beta)_+}
 \wx^{-1+2\nu(1-\alpha)(3-\beta)_+}\ge\frac1{64},
 \wv \le2R_0\wx^{2\nu(1-\alpha)}
 \right\}
 \label{eq:compact-K}\\
 &\cup
 \left\{(x,v):
 \wx^{1-\alpha}\wv ^{\beta-1}\le64
 \right\} \cup
 \left\{(x,v):
 \wv \le\eps^{-\frac1{\beta-1}}, \wx^\alpha\le
 \frac{4C_\chi}{c_\chi}\eps^{-\frac{\beta}{\beta-1}}
 \right\}\\
&\cup
 \left\{(x,v):
 \wx^\alpha\wv ^{\beta-2}\le\frac{128C_\chi}{\eps}
 \right\}.
 \notag
\end{align}
Set
\begin{equation}\label{eq:R-final}
 R:=1+\max_{(x,v)\in\mathcal K}E(x,v).
\end{equation}

The main result of this section, which is \Cref{thm:main-lyapunov} is this range of parameters is, 

\begin{theorem}\label{thm:beta_gt2_alphalt1}
Let $d\ge1$, $0<\alpha<1$, and $\beta>2$.  With the choices
\eqref{eq:explicit-parameters} and the definitions above, the weight $m$ is
positive, tends to $+\infty$ at infinity, has compact sublevel sets, and
satisfies
\begin{equation}\label{eq:main-drift}
 \cL^\star m
 \le C\one_{\{E\le R\}}
 -\frac{c_\chi}{256}
  \eps^{\frac{\beta-2}{\beta-1}}
  \kappa^{1+\frac{2(1-\alpha)}{\alpha}}
  \frac{m}{\{\ln(e+m)\}^{\frac{2(1-\alpha)}{\alpha}}},
\end{equation}
where $R$ is given by \eqref{eq:R-final} and the compact remainder is the
finite, explicitly determined number
\begin{equation}\label{eq:C-final}
 C:=\max_{\{E\le R\}}
 \left[
  \cL^\star m
  +\frac{c_\chi}{256}
   \eps^{\frac{\beta-2}{\beta-1}}
   \kappa^{1+\frac{2(1-\alpha)}{\alpha}}
   \frac{m}{(\ln(e+m))^{\frac{2(1-\alpha)}{\alpha}}}
 \right]_+.
\end{equation}
Consequently, every constant in \eqref{eq:main-drift} is a function of
$(d,\alpha,\beta)$ only, once the universal profile $h$ has been fixed.
\end{theorem}

\subsection{The variable velocity weight}

We start with a collection of identities that will be important later on. As a start, 
\begin{equation*}
\cL^\star q = \sfL^\star q + \sfT q =  \sfL^\star q +v\cdot\nabla_xq - \nabla V\cdot\nabla_vq.
\end{equation*}
Note that, for any $x \in \R^d$,
\begin{align*}
    \sfL^\star q &= \sfL^\star \left( \eta(\tau(x),\cdot)\frac{|\cdot|^2}{2} \right) = \eta(\tau(x),\cdot)\sfL^\star \left( \frac{|\cdot|^2}{2} \right) + \sfL^\star \left( \eta(\tau(x),\cdot)\right) \frac{|\cdot|^2}{2}  +2 v \cdot \nabla_v \eta(\tau(x),\cdot),\\
    &= \eta(\tau(x),\cdot)\left( d-\wv ^{\beta-2}|v|^2\right) + \sfL^\star \left( \eta(\tau(x),\cdot)\right) \frac{|\cdot|^2}{2}  +2 v \cdot \nabla_v \eta(\tau(x),\cdot),
\end{align*}
We have used, 
\begin{equation*}
\sfL^\star \left( \frac{|\cdot|^2}{2} \right)
= d-\wv ^{\beta-2}|v|^2.
\end{equation*}
Given that $\nabla_v \eta = (1-\tau(x))h'\!\left(\frac{\wv }{R_{\tau}}\right) \frac{v}{\wv  R_{\tau}}$,
\begin{align*}
\nabla_v q &= \eta v (\tau(x),\cdot)+  \frac{\vert v\vert^2}{2} \nabla_v \eta(\tau(x),\cdot) \\
&=\eta(\tau(x),\cdot) \, v +  \frac{\vert v\vert^2}{2} (1-\tau(x))h'\!\left(\frac{\wv }{R_{\tau}}\right) \frac{v}{\wv  R_{\tau}}, \\
\sfL^\star \left( \eta(\tau(x),\cdot) \right)& = (1-\tau(x)) \left(\sfL^\star \left(\frac{\wv }{R_{\tau}}\right)  \!h'\left(\frac{\wv }{R_{\tau}}\right) + h''\left(\frac{\wv }{R_{\tau}}\right)\frac{  \left\vert  v\right\vert^2}{\wv ^2 R_{\tau}^2} \right).
\end{align*}
Finally, 
\begin{align*}
    \sfL^\star q    &= \eta(\tau(x),\cdot)\left( d-\wv ^{\beta-2}|v|^2\right) \\
    &+ (1-\tau(x)) \left(\frac{\sfL^\star \left(\wv \right)}{R_{\tau}}  \!h'\left(\frac{\wv }{R_{\tau}}\right) + h''\left(\frac{\wv }{R_{\tau}}\right)\frac{  \left\vert  v\right\vert^2}{\wv ^2 R_{\tau}^2} \right) \frac{|v|^2}{2}  \\
    &+ (1-\tau(x))h'\!\left(\frac{\wv }{R_{\tau}}\right) \frac{2 \vert v \vert^2}{\wv  R_{\tau}},
\end{align*}
Next, 
\begin{align*}
v\cdot\nabla_xq &= \frac{|v|^2}{2}   v\cdot\nabla_x \left(\eta(\tau(x),\cdot) \right) = \frac{|v|^2}{2} \frac{\partial \eta }{\partial \tau}(\tau(x),v)  v\cdot\nabla_x \tau\\
&= \frac{|v|^2}{2} \left(1-h\!\left(\frac{\wv }{R_{\tau}}\right)-(1-\tau(x))h'\!\left(\frac{\wv }{R_{\tau}}\right)\frac{R'(\tau(x))\wv }{R_{\tau}^2}\right) \left( v\cdot\nabla_x \tau \right) \\
\nabla V\cdot\nabla_vq &= \left( \nabla V \cdot v \right) \left( \eta(\tau(x),\cdot) + \frac{|v|^2}{2} h'\!\left(\frac{\wv }{R_{\tau}}\right) \frac{(1-\tau(x))}{\wv  R_{\tau}} \right)
\end{align*}
Conclude with, 
\begin{align*}
 \wv  \sfL^\star(\wv )
 &=\left(d-1-\wv ^{\beta}\right)  +\wv ^{-2}\left(\wv ^{\beta}+1\right)\label{eq:L-wv}
\end{align*}
We shall now proceed with estimates for $\cL^\star q$. Define,
\begin{equation}\label{eq:Dv-def}
D_\tau(v):=\left(\eta_\tau(v) + (1-\tau)h' \left(\frac{\wv }{R_{\tau}}\right)\right)\wv ^\beta
\end{equation}
This is the dissipation scale associated with $q$.  The first term is the usual velocity dissipation generated by the coefficient $\eta_\tau$; the second term is the additional negative drift created by the monotone transition of $h$.

\begin{lemma}\label{lem:q-estimates}
For all $(x,v) \in \R^d \times \R^d$,
\begin{equation}\label{eq:q-gradient}
 |\nabla_vq |^2
 \le(1+\Vert h'\Vert_{\infty})\wv ^{2-\beta}D_{\tau } 
 \le (1+\Vert h'\Vert_{\infty})D_{\tau }.
\end{equation}
Outside the compact set $\mathcal K$ defined in \eqref{eq:compact-K},
\begin{equation}\label{eq:q-full}
 \cL^\star q 
 \le(d+2^{\beta-2})\tau 
   -\frac{D_{\tau }}{32}.
\end{equation}
\end{lemma}

\begin{proof}[{\bf Proof of \Cref{lem:q-estimates}}]

Start with estimates of $\sfL^\star q$ and $|\nabla_vq |^2$.

\smallskip
\noindent{\bf \#  Region $\wv  \le R_\tau$.}  Here $h\left(\frac{\wv }{R_{\tau}}\right) = h'\left(\frac{\wv }{R_{\tau}}\right) = h''\left(\frac{\wv }{R_{\tau}}\right) =0$, hence $q(x,v)=\tau(x) \frac{|v|^2}{2}$.
Thus
\begin{equation*}
\sfL^\star q
=\tau\bigl(d-\wv ^{\beta-2}|v|^2\bigr)=\tau\bigl(d-\wv ^{\beta-2}(\wv ^2 -1)\bigr)
\le \left(d+2^{\beta-2} - \frac{\wv ^\beta}{2}\right) \tau.
\end{equation*}
\begin{equation*}
|\nabla_vq |^2=\tau^2|v|^2.
\end{equation*}

\smallskip
\noindent{\bf \# Region $R_\tau\le \wv  \le2R_\tau$.}  The choice of $R_0$ gives
\begin{equation}\label{eq:annulus-negative-bounds}
 d-\wv ^\beta+\wv ^{\beta-2}\le-\frac34\wv ^\beta,
 \qquad
 \wv \sfL^\star\wv \le-\frac12\wv ^\beta.
\end{equation}
Start from 
\begin{align*}
    \sfL^\star q    &= \eta(\tau(x),\cdot)\left( d-\wv ^{\beta-2}|v|^2\right) \\
    &+ (1-\tau(x)) \left(\frac{\sfL^\star \left(\wv \right)}{R_{\tau}}  \!h'\left(\frac{\wv }{R_{\tau}}\right) + h''\left(\frac{\wv }{R_{\tau}}\right)\frac{  \left\vert  v\right\vert^2}{\wv ^2 R_{\tau}^2} \right) \frac{|v|^2}{2}  \\
    &+ (1-\tau(x))h'\!\left(\frac{\wv }{R_{\tau}}\right) \frac{2 \vert v \vert^2}{\wv  R_{\tau}}\\
    &\leq \eta(\tau(x),\cdot)\left( d-\wv ^{\beta-2}|v|^2\right)+ 2  \Vert h'' \Vert_\infty+  (1-\tau(x)) \left(4+\frac{1}{4}\wv \sfL^\star \left(\wv \right) \right)  \!h'\left(\frac{\wv }{R_{\tau}}\right)\\
    &\leq - \frac34\eta(\tau(x),\cdot) \wv ^\beta + 2  \Vert h'' \Vert_\infty+  (1-\tau(x)) \left(4 -\frac18 \wv ^\beta \right)  \!h'\left(\frac{\wv }{R_{\tau}}\right)
\end{align*}
Since $\eta_\tau \geq \tau$, we have $R_\tau^\beta \eta_\tau > R_0^\beta \tau^{1-\nu \beta} > R_0^\beta$ because $\nu\beta>1$ and $0<\tau\le1$. Hence, all positive terms in the last r.h.s above are absorbed. This proves 
\begin{align*}
    \sfL^\star q &\leq -\frac{1}{16} \left( \eta(\tau(x),\cdot) +  (1-\tau(x)) h'\left(\frac{\wv }{R_{\tau}}\right) \right)\wv ^\beta
\end{align*}
in the annulus. Moreover
\begin{align*}
\left\vert \nabla_v q \right\vert^2& =\left\vert \eta(\tau(x),\cdot) \, +  \frac{\vert v\vert^2}{2} (1-\tau(x))h'\!\left(\frac{\wv }{R_{\tau}}\right) \frac{1}{\wv  R_{\tau}}\right\vert^2 \vert v\vert^2\\
&\leq\left( \eta(\tau(x),\cdot) \, +  (1-\tau(x))h'\!\left(\frac{\wv }{R_{\tau}}\right) \right) \wv ^\beta (1 +\Vert h' \Vert) \vert v\vert^{2-\beta} \\
& \leq \frac{1}{16}\left( \eta(\tau(x),\cdot) \, +  (1-\tau(x))h'\!\left(\frac{\wv }{R_{\tau}}\right) \right) \wv ^\beta, 
\end{align*}
given the value of $R_0$.

\smallskip
\noindent{\bf \# Region $\wv \ge2R_\tau$.}  Here $h=1$, hence $\eta_\tau=1$ and $q=\frac{|v|^2}{2}$.  Thus, the definition of $R_0$ implies,
\begin{equation*}
\sfL^\star q=d-\wv ^{\beta-2}|v|^2\le - \frac12\wv ^\beta,
\qquad
|\nabla_vq|^2=|v|^2\le \wv ^\beta,
\end{equation*}
This completes the proof, since $ \eta(\tau(x),\cdot) \, +  (1-\tau(x))h'\!\left(\frac{\wv }{R_{\tau}}\right) = 1$ in this zone.

Combining the three velocity regions gives, for all
$(x,v)\in \R^d \times \R^d$,
\begin{equation}\label{eq:q-velocity}
 \sfL^\star q
 \le(d+2^{\beta-2})\tau(x)
   -\frac1{16}D_{\tau(x)}(v).
\end{equation}

\smallskip
We now estimate $v\cdot\nabla_xq$ outside of $\mathcal{K}$. We shall split again according to two velocity zones. 

\medskip
\noindent{\bf \# Region $ \wv \le2R_\tau$.}
\medskip

We go back to the full formula
$$v\cdot\nabla_xq = \frac{|v|^2}{2} \left(1-h\!\left(\frac{\wv }{R_{\tau}}\right)-(1-\tau(x))h'\!\left(\frac{\wv }{R_{\tau}}\right)\frac{R'(\tau(x))\wv }{R_{\tau}^2}\right) \left( v\cdot\nabla_x \tau \right)$$
and use $R' = \frac{-\nu}{\tau} R_\tau$, $\vert \nabla_x \tau\vert \leq 2(1- \alpha)\frac{\tau}{\wx}$, $\left(1-h\!\left(\frac{\wv }{R_{\tau}}\right)\right)\tau \leq \eta_\tau$, $2 \nu < 1$ to get
\begin{align*}
\left\vert v\cdot\nabla_xq \right\vert &\leq (1-\alpha) \vert v\vert^3  \frac{\tau}{\wx}\left( \left(1-h\!\left(\frac{\wv }{R_{\tau}}\right)\right) +(1-\tau(x))h'\!\left(\frac{\wv }{R_{\tau}}\right)\frac{\nu R_\tau \cdot 2 R_\tau}{\tau R_{\tau}^2}\right)\\
&\leq\frac{ (1-\alpha) \vert v\vert^3}{\wx}\left( \left(1-h\!\left(\frac{\wv }{R_{\tau}}\right)\right)\tau + 2 \nu (1-\tau(x))h'\!\left(\frac{\wv }{R_{\tau}}\right)\right)\\
&\leq\frac{ (1-\alpha) \vert v\vert^3\wv ^{-\beta}}{\wx}\left( \eta_\tau + (1-\tau(x))h'\!\left(\frac{\wv }{R_{\tau}}\right)\right) \wv ^\beta\\
&\leq(1-\alpha)2^{(3-\beta)_+}R_0^{(3-\beta)_+}
 \wx^{-1+2\nu(1-\alpha)(3-\beta)_+}
  D_{\tau}(v) \leq \frac{1}{64}D_{\tau}(v)
\end{align*}
the last inequality coming from the first component of the definition of $\mathcal{K}$.

\medskip
\noindent{\bf \# Region $ \wv \ge2R_\tau$.}
\medskip
There, the coefficient $\eta(\tau(x),v)$ is equal to one,
so $v\cdot\nabla_xq=0$.  

Finally, it follows that, outside $\mathcal K$,
\begin{equation}\label{eq:q-x-transport}
 |v\cdot\nabla_xq|\le\frac1{64}D_{\tau(x)}(v).
\end{equation}

We finally estimate $\nabla_x V \cdot\nabla_v q$.
\begin{align*}
\left\vert\nabla_x V \cdot\nabla_v q \right\vert
&\leq \vert v\vert \vert \nabla V \vert \left( \eta(\tau(x),\cdot) + \frac{|v|^2}{2} h'\!\left(\frac{\wv }{R_{\tau}}\right) \frac{(1-\tau(x))}{\wv  R_{\tau}} \right)\\
& \leq \wv ^{1-\beta} \vert \nabla V \vert \left( \eta(\tau(x),\cdot) + (1-\tau(x)) h'\!\left(\frac{\wv }{R_{\tau}}\right)  \right) \wv ^\beta \leq \wx^{\alpha-1}\wv ^{1-\beta}D_{\tau}.
\end{align*}
The second component of $\mathcal K$ therefore gives, outside $\mathcal K$,
\begin{equation}\label{eq:q-force-transport}
 |\nabla V\cdot\nabla_vq|
 \le\frac1{64}D_{\tau}(v).
\end{equation}
Combining \eqref{eq:q-velocity}, \eqref{eq:q-x-transport}, and \eqref{eq:q-force-transport} proves \eqref{eq:q-full}.
\end{proof}

\subsection{Estimate on $\Phi$.}
\begin{lemma}\label{lem:macro}
Outside $\mathcal K$,
\begin{equation}\label{eq:phase-drift}
 \cL^\star\Phi
 \le
 -\frac{c_\chi}{128}
  \eps^{\frac{\beta-2}{\beta-1}}\tau
 -\frac{\eps}{128}D_{\tau}(v).
\end{equation}
\end{lemma}

\begin{proof}[{\bf Proof of \Cref{lem:macro}}]
Using, 
\begin{equation*}
\sfT V=v\cdot \nabla_x V,
\qquad
\sfL^\star(\nabla_x V\cdot\chi)=\nabla_x V\cdot\sfL^\star\chi=-\nabla_x V\cdot v,
\end{equation*}
by definition of $\chi$ and
\begin{equation*}
\sfT(\nabla_x V\cdot\chi)=v^\top \text{Hess}(V)\chi-\nabla_x V^\top\nabla_v\chi \nabla_x V,
\end{equation*}
we get
\begin{equation*}\label{eq:L-A}
\cL^\star F
= v^\top\text{Hess}(V)\chi -\nabla_x V^\top\nabla_v\chi \nabla V.
\end{equation*}
Note that, by \eqref{eq:cell-coercivity},
\begin{equation*}
-\nabla_x V^\top\nabla_v\chi \nabla_x V\le -c_\chi \vert \nabla_x V \vert^2 \wv ^{2-\beta},
\end{equation*}
and thus 
\begin{equation*}
\cL^\star F
\leq v^\top\text{Hess}(V)\chi -c_\chi \vert \nabla_x V \vert^2 \wv ^{2-\beta}.
\end{equation*}
Finally,
\begin{equation*}
\cL^\star\Phi = \cL^\star F + \varepsilon \cL^\star q \leq v^\top\text{Hess}(V)\chi  +\eps(d+2^{\beta-2})\tau -c_\chi \vert \nabla_x V \vert^2 \wv ^{2-\beta}
   -\frac{\eps D_{\tau }}{32},
\end{equation*}
recalling \eqref{eq:q-full}. Estimate the latter depending on the size of $\wv $. 

\medskip
\noindent{\bf \# Region $\{\wv \le \varepsilon^{\frac{1}{1-\beta}}\}$.} 
Outside the third component of $\mathcal K$, one has
$\wx^\alpha\ge
 \frac4{c_\chi}\eps^{-\frac{\beta}{\beta-1}}$.
This inequality also implies $\wx\ge\sqrt2$ in view of the choice of $\eps$.  Hence the basic bounds on $V$ give, 
\begin{align*}
 \cL^\star F
 &\le
 \tau(x)\wx^{-\alpha}
 \eps^{-\frac2{\beta-1}}
 -\frac{c_\chi}{2}
  \eps^{\frac{\beta-2}{\beta-1}}\tau
\le
 -\frac{c_\chi}{4}
  \eps^{\frac{\beta-2}{\beta-1}}\tau.
\end{align*}
Using the first smallness condition on $\eps$ in
\eqref{eq:explicit-parameters}, namely
\begin{equation*}
 \eps(d+2^{\beta-2})
 \le\frac{c_\chi}{8}
      \eps^{\frac{\beta-2}{\beta-1}},
\end{equation*}
we obtain
\begin{equation}\label{eq:phase-low}
 \cL^\star\Phi
 \le
 -\frac{c_\chi}{8}
  \eps^{\frac{\beta-2}{\beta-1}}\tau(x)
 -\frac{\eps}{32}D_{\tau(x)}(v).
\end{equation}

\medskip
\noindent{\bf \# Region $\{\wv \ge \varepsilon^{\frac{1}{1-\beta}}\}$.} 

Outside the fourth component of $\mathcal K$, $\wx ^\alpha\wv ^{\beta-2}\ge\frac{128}{\eps}$. Estimate
\begin{equation*}
|v^\top\text{Hess}(V)(x)\chi(v)|
\le C_\chi \tau(x)\wx ^{-\alpha}\wv ^2 = C_\chi \wx ^{-\alpha}\wv ^{2-\beta}\wv ^\beta\tau(x) \leq\frac{\eps}{128}D_{\tau}(v).
\end{equation*}

After discarding the nonpositive coercive term,
\begin{equation*}
 \cL^\star\Phi
 \le\eps(d+2^{\beta-2})\tau
   -\frac{3\eps}{128}D_{\tau}(v).
\end{equation*}
The second smallness condition on $\eps$ in
\eqref{eq:explicit-parameters} gives
\begin{equation*}
 D_{\tau}(v)
 \ge\tau\wv ^\beta
 \ge128(d+2^{\beta-2})\tau,
\end{equation*}
and hence
\begin{equation}\label{eq:phase-high}
 \cL^\star\Phi
 \le-\frac{\eps}{64}D_{\tau(x)}(v).
\end{equation}
Since, in the present region,
\begin{equation*}
 \eps D_{\tau(x)}(v)
 \ge\eps^{-\frac1{\beta-1}}\tau(x)
 \ge c_\chi\eps^{\frac{\beta-2}{\beta-1}}\tau(x),
\end{equation*}
splitting the right-hand side of \eqref{eq:phase-high} into two equal
parts proves \eqref{eq:phase-drift}.

Gathering both regions, and pointing out that the low-velocity estimate
\eqref{eq:phase-low} is stronger than \eqref{eq:phase-drift}, the proof is
complete.
\end{proof}

\subsection{Proof of \Cref{thm:beta_gt2_alphalt1}}

We have
\begin{equation*}
\nabla_v\Phi=(D_v\chi)^\top \nabla_x V+\eps\nabla_vq.
\end{equation*}
Using boundedness of $D_v\chi$, $|
\nabla_x V|^2\le \tau$, and \eqref{eq:q-gradient},
\begin{equation}\label{eq:grad-Phi-bound}
|\nabla_v\Phi|^2
\leq 2 \Vert D_v\chi\Vert_\infty^2\tau+2(1+\Vert h' \Vert_\infty)\eps^2D_\tau.
\end{equation}
The explicit value of $\kappa$ in \eqref{eq:explicit-parameters} therefore
implies, outside $\mathcal K$,
\begin{equation}\label{eq:phase-with-square}
 \cL^\star\Phi+\kappa|\nabla_v\Phi|^2
 \le
 -\frac{c_\chi}{256}
  \eps^{\frac{\beta-2}{\beta-1}}\tau
 -\frac{\eps}{256}D_{\tau}.
\end{equation}
Indeed, $2\kappa\Vert D_v\chi\Vert_\infty^2 \leq \frac{c_\chi}{256}
  \eps^{\frac{\beta-2}{\beta-1}}$ and $2 \kappa(1+\Vert h' \Vert_\infty)\eps^2
 =\frac{\eps}{256}$.

Let us prove that $m$ has compact level sets, and give some bounds on $\Phi$. Since $|
\nabla_x V|^2\le \tau$, and Young's inequality,
\begin{equation*}
|\nabla V(x)\cdot\chi(v)|
\le C_\chi |\nabla V(x)||v|
\le  \frac{\eps}{8}\tau(x)|v|^2+\frac{2C_\chi^2}{\eps}
\le \frac{\eps}{4}q(x,v)+\frac{2C_\chi^2}{\eps},
\end{equation*}
because $q\ge \tau \frac{|v|^2}{2}$. Thus, 
\begin{equation*}
 \Phi_0+ V(x)+\frac{3\eps}{4}q(x,v)-\frac{2C_\chi^2}{\eps} \leq \Phi(x,v) \leq \Phi_0+ V(x)+\frac{5\eps}{4}q(x,v)+\frac{2C_\chi^2}{\eps}.
\end{equation*}
Consequently, for the choice of $\Phi_0$, $\Phi\ge1$, $\Phi\ge V$, and $\Phi$ tends to $+\infty$ at
infinity.  Thus $m=e^{\kappa\Phi}$ is positive and coercive. By the chain rule and \eqref{eq:phase-with-square}, outside
$\mathcal K$,
\begin{align*}
 \cL^\star m =\kappa m\left(\cL^\star\Phi
                 +\kappa|\nabla_v\Phi|^2\right)\le
 -\frac{c_\chi}{256}
  \eps^{\frac{\beta-2}{\beta-1}}
  \kappa m\tau.
\end{align*}
It remains to convert this into a scalar weak Lyapunov form. Since $\Phi\ge V$,
\begin{equation*}
 \tau=\alpha^{-\frac{2(1-\alpha)}{\alpha}}V^{-\frac{2(1-\alpha)}{\alpha}}\ge\Phi^{-\frac{2(1-\alpha)}{\alpha}}.
\end{equation*}
Moreover, $\ln m=\kappa\Phi$ and
$\ln(e+m)\ge\ln m$, so
\begin{equation*}
 \Phi^{-\frac{2(1-\alpha)}{\alpha}}
 =\kappa^{\frac{2(1-\alpha)}{\alpha}}(\ln m)^{-\frac{2(1-\alpha)}{\alpha}}
 \ge\kappa^{\frac{2(1-\alpha)}{\alpha}}(\ln(e+m))^{-\frac{2(1-\alpha)}{\alpha}}.
\end{equation*}
Consequently, outside $\mathcal K$,
\begin{equation*}
 \cL^\star m
 \le
 -\frac{c_\chi}{256}
  \eps^{\frac{\beta-2}{\beta-1}}
  \kappa^{1+\frac{2(1-\alpha)}{\alpha}}
  \frac{m}{(\ln(e+m))^{\frac{2(1-\alpha)}{\alpha}}}.
\end{equation*}
Since $\mathcal K\subset\{E\le R\}$ by \eqref{eq:R-final}, the definition
\eqref{eq:C-final} of $C$ gives the global estimate
\eqref{eq:main-drift} and completes the proof of \Cref{thm:beta_gt2_alphalt1}.
 
\section{The case \texorpdfstring{$\alpha <1$}{} and \texorpdfstring{$\beta < 2$}{}}\label{sec:Lyapunov4}

In this section, we describe the construction of a candidate for a Lyapunov functional when $\alpha < 1$ and $0 <\beta < 2$.

With the convention that a quotient by $(1-\beta)_+=0$ equals $+\infty$, set
\begin{equation}\label{eq:a-choice}
 a:=\frac18\min\left\{
 1,\ \frac{1}{\alpha(3-\beta)},\
 \frac{\frac1\alpha-1}{(1-\beta)_+}
 \right\}, \qquad \omega:=\frac18\min\left\{
 \frac{a\beta}{1+\beta},\
 \frac{(2-\beta)(\frac12-a)}{1+\beta}
 \right\}.
\end{equation}
We provide an explicit value for these just to convince that all conditions that it should satisfy are not void, but actually thinking of them as needing to be $\omega < a$ sufficiently small is enough to understand the proof. 
Define
\begin{equation}\label{eq:z-def}
\Theta(x,v):=\frac{\wv  }{E(x,v)^a}.
\end{equation}
Observe that given that $E$ is large, if $\Theta$ is large then so does $v$. On the contrary, is $\Theta$ is small then $v$ is bounded, so $x$ is large. This will be quantified later on. 

We take an increasing function $\zeta \in C^2((0,\infty))$ such that
\begin{equation*}
 0<\zeta\le1,
 \qquad \zeta(s)=e^{-\frac{1}{s}}\quad(0<s\le1),
 \qquad \zeta(s)=1\quad(s\ge 2),
 \qquad \zeta'\ge0.
\end{equation*}
On $(0,1]$,
\begin{equation}\label{eq:zeta-derivatives}
 \zeta'(s)= s^{-2}\zeta(s),
 \qquad
 \zeta''(s)=\left(s^{-4}-2 s^{-3}\right)\zeta(s),
\end{equation}
and hence
\begin{equation*}
 (\zeta''(s))_+\le s^{-4}\zeta(s),\qquad 0<s\le1.
\end{equation*}
On the fixed interval $[1,2]$ we choose the interpolation so that
\begin{equation*}
 |\zeta'(s)|+|\zeta''(s)|\le C,
 \qquad 1\le s\le 2.
\end{equation*}

We are now ready to define $m$ properly. Start with setting
\begin{equation}\label{eq:A-B}
F:=V^\frac{\beta}{2}+\nabla(V^\frac{\beta}{2})\cdot\chi,
\end{equation}
where $\chi$ is again given by \Cref{thm:cell}, and the interpolated phase
\begin{equation*}
 \Phi:=\Phi_0 + (1-\zeta(\Theta))F + \zeta(\Theta)E^\frac{\beta}{2},
 \qquad
 m:=e^{\kappa\Phi}.
\end{equation*}
The following theorem establishes \Cref{thm:main-lyapunov} in the present parameter regime. \begin{theorem}\label{thm:Lyap-alphalt1-betalt2}
There exists $\kappa_0>0$ such that, for each fixed
\begin{equation*}
0<\kappa\le\kappa_0,
\end{equation*}
there are constants $c>0$, $C>0$, and $R>0$ for which $m=e^{\kappa\Phi}$ has compact sublevel sets and satisfies,
\begin{equation}\label{eq:weak-main}
 \cLs m
 \le C\one_{\{E\le R\}}
 -c \kappa^{1+\sigma} \frac{m}{(\ln(e+m))^{\sigma}},
\end{equation}
with
\begin{equation}\label{eq:sigma}
 \sigma=\frac{1-\frac{\beta}{2}+\frac{2(1-\alpha)}{\alpha}}{\frac{\beta}{2}}.
\end{equation}
\end{theorem}

\subsection{Elementary estimates for \texorpdfstring{$\Theta$}{\Theta}}

We shall first record some estimates that will be used later on during the proofs. 

\begin{lemma}\label{lem:infoTheta}
There is $R_\Theta>0$ such that, whenever $E \geq R_\Theta$,
\begin{enumerate}
    \item if $\Theta\le 2$ then $E\asymp V$. 
\item whenever $\Theta\le 2$ and $a<\frac12$,
\begin{equation*}
|\nabla_v\Theta|\lesssim E^{-a}, \qquad  |\nabla_v^2\Theta|\lesssim \Theta^{-1}E^{-2a}.    
\end{equation*}
\item There are $c,R>0$ such that, whenever $E\ge R$ and $E^{-\omega}\le \Theta \le2$,
\begin{equation*}
 \cLs \Theta\le- \frac{\Theta^{\beta-1}}{4}E^{-a(2-\beta)}.
\end{equation*}
\end{enumerate}   
\end{lemma}

\begin{proof}[{\bf Proof of \Cref{lem:infoTheta}}]
The first identity is rather direct, given that if $\Theta\le 2$, $\wv ^2 \leq 2^2 E^{2a} \leq 4 R^{2a-1} E$, when $E \geq R$ since $a < \frac12$. As a consequence, $E(x,v) = \frac{\vert v \vert^2}{2} + V(x) \leq 2 R^{2a-1} E + V(x)$ and thus $$V \leq E \leq \left( 1 - 2 R^{2a-1} \right)^{-1} V,$$
when $E \geq R > 2^{\frac{1}{1-2a}}$.

Since $\Theta=\wv   E^{-a}$,  
\begin{equation}\label{eq:grad-z-exact}
 \nabla_v\Theta=E^{-a}\left(1-a \frac{\wv ^2}{E}\right)\frac{v}{\wv  }.
\end{equation}
Since, when $\Theta\le 2$, $\wv  \leq 2 E^a$, or equivalently $\wv ^2 \leq 2^2 E^{2a}$, we have $a\frac{\wv ^2}{E}\leq 4 a R^{2a-1} < 2R^{2a-1}$ and thus
\begin{equation}\label{eq:grad-z}
 |\nabla_v\Theta|\le E^{-a}
\end{equation}
when $E \geq R > 2^{\frac{1}{1-2a}}$. The second derivative is
\begin{align}\label{eq:hess-z-exact}
 E^{a}\nabla_v^2\Theta &= \nabla_v^2\wv  
 -a\left(2\frac{v\otimes v}{\wv ^2 }
 +I \right)\wv  E^{-1} +a(a+1)\wv  E^{-2}v\otimes v \notag\\
 &= \nabla_v^2\wv  
 -a\left(2\frac{v\otimes v}{\wv ^2 }
 +I \right)\Theta E^{a-1} +a(a+1)\Theta^3 E^{3a-2}\frac{v\otimes v}{\wv ^2 }\notag\\
 &= \Theta^{-1}E^{-a}\left(\Theta E^{a}\nabla_v^2\wv  
 -a\left(2\frac{v\otimes v}{\wv ^2 }
 +I \right) \Theta^2 E^{2a-1} +a(a+1)\Theta^4 E^{4a-2}\frac{v\otimes v}{\wv ^2 }\right).
\end{align}
Since $\nabla_v^2\wv $ is of order $\wv ^{-1} =\Theta^{-1}E^{-a}$, on $\Theta\le 2$,
\begin{equation}\label{eq:hess-z}
 |\nabla_v^2\Theta|\lesssim \Theta^{-1}E^{-2a},
\end{equation}
when $E \geq R$ is large.


The exact formula for $\cLs \Theta$ is
\begin{equation}\label{eq:Lz-exact}
\begin{aligned}
 \cLs \Theta={}& E^{-a}\left(\frac d{\wv  }- \left( 1 + \wv ^\beta\right) \frac{|v|^2}{\wv  ^3}- \frac{v \cdot \nabla V}{\wv  }\right)\\
 &-a\wv  E^{-a-1}\left(d +\left( 2 -\wv  ^{\beta}\right) \frac{|v|^2}{\wv  ^2}\right)
 +a(a+1)\wv  E^{-a-2}|v|^2.
\end{aligned}
\end{equation}
On the strip $E^{-\omega}\le \Theta\le 2$, for $E$ large, $v$ is large since $\wv  = \Theta E^a \geq E^{a- \omega}$. From all this, we deduce,
\begin{equation*}
E^a \cLs \Theta \leq\left(\frac {d}{\wv  } + \vert \nabla V \vert - \wv ^{\beta-3}|v|^2\right) +a\wv ^{\beta+1} E^{-1}
 +a(a+1)\wv ^3 E^{-2}
\end{equation*}
Observe that 
\begin{align*}
&\wv ^{\beta+1} E^{-1} \leq \Theta^{\beta+1} E^{a(\beta+1)-1} = \Theta^{2} E^{2a-1} \Theta^{\beta-1}E^{-a(1-\beta)} \leq 4 R^{2a-1} \Theta^{\beta-1}E^{-a(1-\beta)},\\
&\wv ^3 E^{-2} = \Theta^{4-\beta}E^{a(4-\beta)-2} \Theta^{\beta-1}E^{-a(1-\beta)} \leq 2^{4-\beta}R^{a(4-\beta)-2} \Theta^{\beta-1}E^{-a(1-\beta)},\\
&\wv ^{\beta-3}|v|^2 =\wv ^{\beta-1} \frac{|v|^2}{\wv ^2} = \wv ^{\beta-1} \left(1-\frac{1}{\wv ^2} \right) \geq \frac{\wv ^{\beta-1}}2 = \frac{\Theta^{\beta-1}}2E^{-a(1-\beta)},
\end{align*}
\begin{align*}
\vert \nabla V \vert &=  \left(\vert \nabla V \vert \Theta^{1-\beta}E^{a(1-\beta)} \right) \Theta^{\beta-1}E^{-a(1-\beta)} \leq \left(\vert \nabla V \vert \Theta^{1-\beta}V^{a(1-\beta)} \right) \Theta^{\beta-1}E^{-a(1-\beta)}   \\
&\leq \left( E^{1-\frac{1}{\alpha}+a(1-\beta)+\omega(\beta-1)_+}\right) \Theta^{\beta-1}E^{-a(1-\beta)} 
\end{align*}

the latter being true when $\wv ^2$ is larger that $\frac12$.
As a consequence, 
\begin{equation}
E^a \cLs \Theta \leq - \frac{\Theta^{\beta-1}}4E^{-a(1-\beta)},
\end{equation}
if $E \geq R$ and $R$ is large enough and $a$ and $\omega$ are well-chosen.
   
\end{proof}

\subsection{Estimates on $E^\frac{\beta}{2} - F$}

In view of the latter computations, 
\begin{lemma}\label{lem:Etheta-F}
When $E$ is large enough, on the strip $E^{-\omega}\le \Theta\le 2$, one has, 
\begin{equation*}
 E^\frac{\beta}{2} - F \asymp \Theta^2E^{\frac{\beta}{2}-1+2a}, \quad |\nabla_v(E^\frac{\beta}{2} - F)|
 \le \wv  E^{\frac{\beta}{2}-1} + \wv  ^{2-\beta}V^{\frac{\beta}{2}-\frac{1}{\alpha}}.
\end{equation*}
\end{lemma}

\begin{proof}[{\bf Proof of \Cref{lem:Etheta-F}}]
Recall that $$E^\frac{\beta}{2} - F = \left(\frac{\vert v \vert^2}{2} + V\right)^\frac{\beta}{2} -V^\frac{\beta}{2} -\nabla(V^\frac{\beta}{2})\cdot\chi =V^\frac{\beta}{2} \left(\left(\frac{\vert v \vert^2}{2V} + 1\right)^\frac{\beta}{2}-1\right) -\nabla(V^\frac{\beta}{2})\cdot\chi .$$
We shall estimate both parts separately. On the strip $E^{-\omega}\le \Theta\le 2$, $v$ is large. Recalling $\wv  = \Theta E^a$ and $V \asymp E$, we obtain $\frac{\vert v \vert^2}{2V} \asymp \Theta^2 E^{2a-1}$. Since the latter is small, 
$$V^\frac{\beta}{2} \left(\left(\frac{\vert v \vert^2}{2V} + 1\right)^\frac{\beta}{2}-1\right) \asymp \Theta^2 E^{\frac{\beta}{2} +2a-1}.$$
The term $
\nabla(V^\frac{\beta}{2})\cdot\chi$ is smaller, since, recalling \Cref{thm:cell},
\begin{align*}
    \left\vert \nabla(V^\frac{\beta}{2})\cdot\chi\right\vert &\leq V^{\frac{\beta}{2}-1}\left\vert \nabla V \right\vert \left\vert \chi\right\vert \leq V^{\frac{\beta}{2} - 1 +1 - \frac{1}{\alpha}} \wv ^{3-\beta} \\
    &\lesssim E^{\frac{\beta}{2} - \frac{1}{\alpha}}\wv ^{3-\beta}= \left(E^{1- \frac{1}{\alpha} +a(1- \beta)}\Theta^{1-\beta} \right)\Theta^2 E^{\frac{\beta}{2} +2a-1} \leq \left(E^{1- \frac{1}{\alpha} + a(1-\beta) - \omega(1-\beta)_+} \right)\Theta^2 E^{\frac{\beta}{2} +2a-1}.
\end{align*}
When $a$ and $\omega$ are suitably chosen, the above bracket is small when $E$ is large. Furthermore,
\begin{align*}
\nabla_v(E^\frac{\beta}{2} - F) &= \nabla_v \left( E^\frac{\beta}{2} -V^\frac{\beta}{2} -\nabla(V^\frac{\beta}{2})\cdot\chi \right)
= \frac{\beta}{2} E^{\frac{\beta}{2}-1} v - \frac{\beta}{2} V^{\frac{\beta}{2} - 1} D_v \chi  \nabla V
\end{align*}
and thus 
\begin{align*}
\left\vert \nabla_v(E^\frac{\beta}{2} - F)  \right\vert \lesssim \wv  E^{\frac{\beta}{2}-1} + \wv  ^{2-\beta}V^{\frac{\beta}{2}-\frac{1}{\alpha}}.
\end{align*}
\end{proof}

\begin{lemma}\label{lem:estPhi}
Taking $\kappa$ small, when $E$ and $v$ are large,
\begin{equation}\label{eq:B-local-negative}
 \cLs E^\frac{\beta}{2}+4\kappa|\nabla_vE^\frac{\beta}{2}|^2\le - \frac{\beta}{4}E^{\frac{\beta}{2}-1}\wv  ^\beta.
\end{equation}
When $E$ is large, $\Theta \leq 2$ provided $\kappa$ is small enough,
\begin{equation}\label{eq:A-good}
 \cLs F + 4 \kappa \left\vert \nabla_v F \right\vert^2 \le -\frac14c_{\chi} V^{\frac{\beta}{2}-1}\vert\nabla V\vert^2 \wv ^{2-\beta}.
\end{equation}

\end{lemma}

\begin{proof}[{\bf Proof of \Cref{lem:estPhi}}]

Now, 
\begin{equation*}
 \cLs E^\frac{\beta}{2}=\frac{\beta}{2} E^{\frac{\beta}{2}-1}\big(d-\wv  ^{\beta-2}|v|^2\big)
 +\frac{\beta}{2}\left(\frac{\beta}{2}-1\right)E^{\frac{\beta}{2}-2}|v|^2.
\end{equation*}
Moreover, since $|\nabla_vE^\frac{\beta}{2}|^2=\frac{\beta^2}{4}E^{\beta-2}|v|^2$,
\begin{align*}
 \cLs E^\frac{\beta}{2}+4\kappa|\nabla_vE^\frac{\beta}{2}|^2 & =\frac{\beta}{2} E^{\frac{\beta}{2}-1}\big(d-\wv  ^{\beta-2}|v|^2\big)
+\frac{\beta}{2}\left(\frac{\beta}{2}-1\right)E^{\frac{\beta}{2}-2}|v|^2 +  \kappa \beta^2E^{\beta-2}|v|^2\notag\\
 &\le \frac{\beta}{2} E^{\frac{\beta}{2} -1} \left( d-\wv  ^{\beta-2}|v|^2 + 2 \beta \kappa E^{\frac{\beta}{2}-1}|v|^2\right).
\end{align*}
Since $\frac{\beta}{2} \leq 1$, $E^{\frac{\beta}{2}-1} < 2^{1-\frac{\beta}{2}}\vert v \vert^{2(\frac{\beta}{2}-1)}$, 
\begin{align*}
 \cLs E^\frac{\beta}{2}+4\kappa|\nabla_vE^\frac{\beta}{2}|^2 &\le \frac{\beta}{2} E^{\frac{\beta}{2} -1} \left( d + \left(2^{3-\frac{\beta}{2}} \kappa \frac{\beta}{2} \vert v \vert^{2(\frac{\beta}{2}-1)}-\wv  ^{\beta-2} \right)|v|^2\right).
\end{align*}
Taking $\kappa$ small, when $E$ and $v$ are large,
\begin{equation*}
 \cLs E^\frac{\beta}{2}+4\kappa|\nabla_vE^\frac{\beta}{2}|^2\le - \frac{\frac{\beta}{2}}{2}E^{\frac{\beta}{2}-1}\wv  ^\beta.
\end{equation*}
Observe that
\begin{align*}
 \cLs F &= \sfT F + \sfL^\star F = \sfT (V^\frac{\beta}{2}+\nabla(V^\frac{\beta}{2})\cdot\chi) + \sfL^\star (V^\frac{\beta}{2}+\nabla(V^\frac{\beta}{2})\cdot\chi) \\
 &= v \cdot \nabla(V^\frac{\beta}{2})+\sfT (\nabla(V^\frac{\beta}{2})\cdot\chi) + \nabla(V^\frac{\beta}{2})\cdot\sfL^\star\chi \\
 &= \frac{\beta}{2} V^{\frac{\beta}{2}-1} v^\top D_x^2V \chi + \frac{\beta}{2}\left(\frac{\beta}{2}-1\right)V^{\frac{\beta}{2}-2} (v \cdot \nabla V)(\chi \cdot \nabla V) - \frac{\beta}{2} V^{\frac{\beta}{2}-1}\nabla V^\top(D_v\chi)\nabla V .
\end{align*}
All latter terms may be estimated, the leading order term being the coercive one,
\begin{align*}
    V^{\frac{\beta}{2}-1}\nabla V^\top(D_v\chi)\nabla V \geq c_{\chi} V^{\frac{\beta}{2}-1}\vert\nabla V\vert^2 \wv ^{2-\beta}
\end{align*}
The other terms are lower order since, using $V \asymp E$ and $\wv  \leq 2 E^a$ on $\Theta\le 2$,
\begin{align*}
    &\left\vert V^{\frac{\beta}{2}-1} v^\top D_x^2V \chi \right\vert \lesssim \left(V^{1-\frac{2}{\alpha}} \wv ^2 \vert\nabla V\vert^{-2} \right)V^{\frac{\beta}{2}-1}\vert\nabla V\vert^2\wv ^{2-\beta} \lesssim \left(\wv ^2 V^{-1}\right)V^{\frac{\beta}{2}-1}\vert\nabla V\vert^2\wv ^{2-\beta}\\
    &\left\vert V^{\frac{\beta}{2}-2} (v \cdot \nabla V)(\chi \cdot \nabla V) \right\vert \lesssim \wv ^{4-\beta} V^{\frac{\beta}{2}-2}\vert\nabla V\vert^2 = \left( V^{-1} \wv ^2 \right)V^{\frac{\beta}{2}-1}\vert\nabla V\vert^2\wv ^{2-\beta}
\end{align*}
and $\wv ^2 V^{-1}$ is smaller than $E^{2a-1}$ which tends to zero. As a consequence, when $E$ is large, on $\Theta \leq 2$, 
\begin{equation*}
 \cLs F\le -\frac12c_{\chi} V^{\frac{\beta}{2}-1}\vert\nabla V\vert^2 \wv ^{2-\beta}.
\end{equation*}
Also, 
\begin{align*}
\left\vert \nabla_v F \right\vert^2 = \frac{\beta^2}{4} V^{\beta-2}\left\vert D_v \chi  \nabla V \right\vert^2 \lesssim\frac{\beta^2}{4} V^{\frac{\beta}{2} - 1} \wv ^{2-\beta}\left( V^{\frac{\beta}{2} - 1}\wv ^{2-\beta} \left\vert \nabla V \right\vert^2\right) 
\end{align*}
Since $V^{\frac{\beta}{2} - 1} \wv ^{2-\beta}  \asymp E^{\frac{\beta}{2} - 1} \Theta^{2-\beta} E^{a(2-\beta)} \leq 2^{2-\beta} E^{(a-\frac12)(2-\beta)}$, the term $\left\vert \nabla_v F \right\vert^2$ can be absorbed by $\cLs F$ and thus can be absorbed when $E$ (thus $V$) is large.

Consequently, 
\begin{equation*}
 \cLs F + 4 \kappa \left\vert \nabla_v F \right\vert^2 \le -\frac14c_{\chi} V^{\frac{\beta}{2}-1}\vert\nabla V\vert^2 \wv ^{2-\beta}
\end{equation*}
when $E$ is large, provided $\kappa$ is small enough.
    
\end{proof}

\subsection{Estimates on $\Phi$}

\begin{proposition}[Lyapunov estimate on $\Phi$]\label{prop:LyapPhi}
There exists $\kappa_0>0$ such that, for each fixed
\begin{equation*}
0<\kappa\le\kappa_0,
\end{equation*}
there are constants $c>0$, $C>0$, and $R>0$ for which
\begin{equation}\label{eq:phase-final}
\cLs\Phi+2\kappa|\nabla_v\Phi|^2
 \le C\one_{\{E\le R\}}-c\big((1-\zeta(\Theta))V^{\frac{\beta}{2}-1}|\nabla V|^2\wv  ^{2-\beta} +\zeta(\Theta) E^{\frac{\beta}{2}-1}\wv  ^\beta \big).
\end{equation} 
\end{proposition}

\begin{proof}[{\bf Proof of \Cref{prop:LyapPhi}}]

Start by 
\begin{align*}\label{eq:grad-Phi-exact}
\nabla_v \Phi =(1-\zeta(\Theta))\nabla_vF + \zeta(\Theta)\nabla_v E^\frac{\beta}{2} +\zeta'(\Theta) (E^\frac{\beta}{2} - F) \nabla_v \Theta.
\end{align*}
Since $0\le \zeta\le1$, convexity gives
\begin{equation*}
 |(1-\zeta(\Theta))\nabla_v F+\zeta(\Theta)\nabla_vE^\frac{\beta}{2}|^2\le (1-\zeta(\Theta))|\nabla_v F|^2+\zeta(\Theta)|\nabla_vE^\frac{\beta}{2}|^2.
\end{equation*}
Using Young's inequality, we obtain:
\begin{equation}\label{eq:grad-square-split}
|\nabla_v\Phi|^2
 \le 2(1-\zeta)|\nabla_vF|^2
 +2\zeta |\nabla_vE^\frac{\beta}{2}|^2
 +2 |\zeta'(\Theta)(E^\frac{\beta}{2} - F)\nabla_v \Theta|^2.
\end{equation}

We first compute
\begin{align}\label{eq:LPhi-exact}
\cLs\Phi &= \cLs\left( (1-\zeta(\Theta))F \right) + \cLs \left( \zeta(\Theta)E^\frac{\beta}{2} \right)\notag\\
  &=(1-\zeta(\Theta))\cLs F - 2 \zeta'(\Theta) \nabla_v \Theta \cdot \nabla_v F -\cLs (\zeta(\Theta) ) F \notag\\
  &+ \zeta(\Theta) \cLs ( E^\frac{\beta}{2} ) + 2 \zeta'(\Theta) \nabla_v \Theta \cdot \nabla_v E^\frac{\beta}{2}  + \cLs \left( \zeta(\Theta)\right)E^\frac{\beta}{2}\notag\\
  &=(1-\zeta(\Theta))\cLs F + \zeta(\Theta) \cLs ( E^\frac{\beta}{2} )
   + 2 \zeta'(\Theta) \nabla_v \Theta \cdot \nabla_v \left( E^\frac{\beta}{2} - F \right) + \cLs \left( \zeta(\Theta)\right) \left( E^\frac{\beta}{2} - F\right)\notag\\
  &=(1-\zeta(\Theta))\cLs F + \zeta(\Theta) \cLs ( E^\frac{\beta}{2} )\notag\\
   &+ \zeta'(\Theta) \left(2\nabla_v \Theta \cdot \nabla_v ( E^\frac{\beta}{2} - F ) + (\cLs\Theta) ( E^\frac{\beta}{2} - F)\right) +  \zeta''(\Theta) \vert \nabla_v \Theta \vert^2 ( E^\frac{\beta}{2} - F) 
\end{align}
where we have used
$$\cLs \left( \zeta(\Theta)\right) = \zeta'(\Theta) \cLs\Theta + \zeta''(\Theta) \vert \nabla_v \Theta \vert^2.$$
Combining previous expressions, we get that
\begin{align*}\label{eq:LPhi-exact}
\cLs\Phi + 2 \kappa \left\vert \nabla_v \Phi\right\vert^2 &\leq (1-\zeta(\Theta)) \left(\cLs F + 4 \kappa \left\vert \nabla_v F\right\vert^2\right)+ \zeta(\Theta)\left( \cLs ( E^\frac{\beta}{2} ) + 4 \kappa \left\vert \nabla_v E^\frac{\beta}{2}\right\vert^2\right)\notag\\
   &+ 2\zeta'(\Theta) \nabla_v \Theta \cdot \nabla_v ( E^\frac{\beta}{2} - F ) + \zeta'(\Theta) (\cLs\Theta) ( E^\frac{\beta}{2} - F) \\
   &+  \zeta''(\Theta) \vert \nabla_v \Theta \vert^2 ( E^\frac{\beta}{2} - F) + 4 \kappa |\zeta'(\Theta)(E^\frac{\beta}{2} - F)\nabla_v \Theta|^2
\end{align*}
We split the phase space into several zones to estimate the dissipation, assuming that $E\geq R$, for $R$ sufficiently large to be set along the proof.

\medskip

\noindent{\bf \# Region $\Theta\le E^{-\omega}$.}
Here $\Theta\le1$, so $\zeta(\Theta)\le e^{- E^\omega}$. In this region, all the terms having a $\zeta$, $\zeta'$, or $\zeta''$ are very small, so they are negligible with respect to $V^{\frac{\beta}{2}-1}|\nabla V|^2\wv  ^{2-\beta}$. 
  
Since $1-\zeta(\Theta)\ge \frac12$ in this region, using \Cref{lem:estPhi}, we have
\begin{equation}\label{eq:region-I}
 \cLs\Phi+2\kappa|\nabla_v\Phi|^2\le -\frac18 V^{\frac{\beta}{2}-1}|\nabla V|^2\wv  ^{2-\beta} ,
\end{equation}
after discarding all the small remainder terms detailed above. 
\medskip

\noindent{\bf \# Region $E^{-\omega}\le \Theta\le1$.}
Here $\wv  \ge E^{a-\omega}$, so \eqref{eq:B-local-negative} applies.  The term $\zeta'(\Theta)(\cLs \Theta)(E^\frac{\beta}{2} - F)$ is negative and, using \eqref{eq:zeta-derivatives}, \Cref{lem:infoTheta} and \Cref{lem:Etheta-F},
\begin{equation*}\label{eq:negative-transition}
 \zeta'(\Theta)(\cLs \Theta)(E^\frac{\beta}{2} - F)
 \le -\frac14 \zeta(\Theta)\Theta^{\beta-1}E^{\frac{\beta}{2}-1+a\beta}.
\end{equation*}
The positive part of the $\zeta''$-term is
\begin{equation*}\label{eq:zeta-second-est}
(\zeta''(\Theta))_+|\nabla_v\Theta|^2(E^\frac{\beta}{2} - F)
 \lesssim   \zeta(\Theta)\Theta^{-2}E^{\frac{\beta}{2}-1} =\left( \Theta^{-1-\beta}E^{-a\beta}\right) \zeta(\Theta)\Theta^{\beta-1}E^{\frac{\beta}{2}-1+a\beta}.
\end{equation*}
The bracket term might be estimated as follows. 
\begin{equation*}
\Theta^{-1-\beta}E^{-a\beta}
 \le E^{\omega(1+\beta)-a\beta},
\end{equation*}
which is small when $E$ is large, since $\omega(1+\beta)-a\beta<0$. Thus, $(\zeta''(\Theta))_+|\nabla_v\Theta|^2(E^\frac{\beta}{2} - F)$ is absorbed by $\zeta'(\Theta)(\cLs \Theta)(E^\frac{\beta}{2} - F)$.

The following cross term satisfies, using \Cref{lem:infoTheta} and \Cref{lem:Etheta-F}, recalling that $E \asymp V$ in this zone,
\begin{align*}\label{eq:cross-est} 2|\zeta'(\Theta)\nabla_v\Theta\cdot\nabla_v(E^\frac{\beta}{2} - F)|
 &\lesssim C\zeta(\Theta)\Theta^{-2}E^{-a}
 \left(E^{\frac{\beta}{2}-1}\wv  +E^{\frac{\beta}{2}- \frac{1}{\alpha}}  \wv ^{2-\beta}\right)\\
 &\lesssim C\zeta(\Theta)
 \left(\Theta^{-\beta}E^{-a\beta} +E^{1- \frac{1}{\alpha}}E^{a(1-2\beta)}\Theta^{1-2\beta}\right)\Theta^{\beta-1}E^{\frac{\beta}{2}-1+a\beta}
 .
\end{align*}
The term $\Theta^{-\beta}E^{-a\beta}$ is small when $E$ is large, since $\omega < a$. For the term $\Theta^{1-2\beta}E^{1-\frac{1}{\alpha}+a(1-2\beta)}$, it is also small when $E$ is large (since $1-\frac{1}{\alpha}+(a-\omega)(1-2\beta)_+<0$) . 

Finally, the last term in \eqref{eq:grad-square-split} satisfies
\begin{equation*}
 4\kappa|\zeta'(\Theta)(E^\frac{\beta}{2} - F)\nabla_v\Theta|^2
 \lesssim \kappa \zeta(\Theta)^2E^{\beta-2+2a} = \left(\kappa\zeta(\Theta)\Theta^{1-\beta}E^{\frac{\beta}{2}-1+a(2-\beta)}\right)\zeta(\Theta)\Theta^{\beta-1}E^{\frac{\beta}{2}-1+a\beta}.
\end{equation*}
The latter bracket term is small when $E$ is large since $\frac{\beta}{2}-1+a(2-\beta)- \omega(1-\beta)_+ < 0$. 

We obtain
\begin{equation}\label{eq:region-II}
\cLs\Phi+2\kappa|\nabla_v\Phi|^2
 \le -c\big((1-\zeta(\Theta))V^{\frac{\beta}{2}-1}|\nabla V|^2\wv  ^{2-\beta} +\zeta(\Theta) E^{\frac{\beta}{2}-1}\wv  ^\beta \big),
\end{equation}
in this zone.
\medskip

\noindent{\bf \# Region $1\le \Theta\le 2$.}

Here $E\asymp V$, $\wv  \asymp E^a$, and $\zeta\ge e^{-1}$.  Hence \eqref{eq:B-local-negative} gives a negative energy contribution.  Since $|\zeta'|+|\zeta''|\le C_\zeta$, all transition terms are powers of $E$.  More precisely,
\begin{equation*}
|\zeta''||\nabla_v\Theta|^2(E^\frac{\beta}{2} - F)\lesssim  E^{\frac{\beta}{2}-1}, \qquad |\zeta'\nabla_v\Theta\cdot E^{\frac{\beta}{2}-1}v|\lesssim E^{\frac{\beta}{2}-1}.
\end{equation*}
\begin{equation*}
    \left\vert 2\zeta'(\Theta) \nabla_v \Theta \cdot \nabla_v ( E^\frac{\beta}{2} - F ) \right\vert \lesssim  E^{1-\frac{1}{\alpha}+a(1-2\beta)}E^{\frac{\beta}{2}-1}\wv  ^\beta
\end{equation*}
The squared term satisfies
\begin{equation*}
4\kappa|\zeta'(E^\frac{\beta}{2} - F)\nabla_v\Theta|^2\le C\kappa E^{\beta-2+2a} \lesssim \left(\kappa E^{\frac{\beta}{2}-1+a(2-\beta)}\right) E^{\frac{\beta}{2}-1}\wv  ^\beta ,
\end{equation*}
Consequently, after increasing $R$,
\begin{equation}\label{eq:region-III}
 \cLs\Phi+2\kappa|\nabla_v\Phi|^2
 \le -c\big((1-\zeta(\Theta))V^{\frac{\beta}{2}-1}|\nabla V|^2\wv  ^{2-\beta} +\zeta(\Theta) E^{\frac{\beta}{2}-1}\wv  ^\beta \big).
\end{equation}

\medskip

\noindent{\bf \# Region $\Theta\ge 2$.}
Here $\zeta(\Theta)=1$ and $\Phi=E^\frac{\beta}{2}$.  Since $\wv  \ge 2E^a$, \eqref{eq:B-local-negative} applies for $E$ large. Therefore
\begin{equation}\label{eq:region-IV}
 \cLs\Phi+2\kappa|\nabla_v\Phi|^2
 \le - \frac{\beta}{4}E^{\frac{\beta}{2}-1}\wv  ^\beta.
\end{equation}
Combining \eqref{eq:region-I}, \eqref{eq:region-II}, \eqref{eq:region-III}, and \eqref{eq:region-IV}, we obtain \eqref{eq:phase-final}.

\end{proof}

\subsection{Proof of the Lyapunov estimate}

We are now ready to prove the main result. 

\begin{proof}[{\bf Proof of \Cref{thm:Lyap-alphalt1-betalt2}}]
For $m=e^{\kappa\Phi}$,
\begin{equation}\label{eq:chain-m}
 \frac{\cLs m}{m}=\kappa\cLs\Phi+\kappa^2|\nabla_v\Phi|^2
 =\kappa\left(\cLs\Phi+2\kappa|\nabla_v\Phi|^2\right)-\kappa^2|\nabla_v\Phi|^2.
\end{equation}
Thus \eqref{eq:phase-final} implies 
\begin{equation}\label{eq:strong-main}
 \cLs m
 \le C\one_{\{E\le R\}}
 -c \kappa \big((1-\zeta(\Theta))V^{\frac{\beta}{2}-1}|\nabla V|^2\wv ^{2-\beta}+\zeta(\Theta)E^{\frac{\beta}{2}-1}\wv ^\beta\big)\,m,
\end{equation}
 
The compactness of the sublevel sets comes as follows.  On $\Theta\le 2$, $E\asymp V$ and $|
\nabla(V^{\frac{\beta}{2}})\cdot\chi| \lesssim V^{\frac{\beta}{2}-\frac{1}{\alpha}}\wv^{3-\beta} \lesssim E^{\frac{\beta}{2}-\frac{1}{\alpha}}E^{a(3-\beta)}=o(V^\frac{\beta}{2})$, so $F\asymp V^\frac{\beta}{2}$.  On $\Theta\ge 2$, $\Phi=E^\frac{\beta}{2}$.  Hence $\Phi\to\infty$, and $m=e^{\kappa\Phi}$ is coercive. 

It remains to pass to a scalar weak Lyapunov form. On $\Theta\ge 2$, $\Phi\asymp E^\frac{\beta}{2}$ and $\wv  \gtrsim E^a$, hence
\begin{equation*}
 E^{\frac{\beta}{2}-1}\wv  ^\beta =E^{\frac{\beta}{2}-1}\wv  ^\beta\gtrsim E^{\frac{\beta}{2}-1+a\beta}\gtrsim E^{-\sigma\frac{\beta}{2}} \gtrsim\Phi^{-\sigma} \asymp \kappa^\sigma \left(\ln (e+m)\right)^{-\sigma}.
\end{equation*}
Moreover, $\zeta(\Theta) =1$, so from \eqref{eq:strong-main},
\begin{equation}
 \cLs m
 \le C\one_{\{E\le R\}}
 -c \kappa \big(E^{\frac{\beta}{2}-1}\wv ^\beta\big)\,m \leq C\one_{\{E\le R\}}
 -c \kappa^{1+\sigma} \,\frac{m}{ \left(\ln (e+m)\right)^{\sigma}}.
\end{equation}
On $\Theta\le 2$, notice that if $\zeta(\Theta)\le \frac12$, the term $(1-\zeta)V^{\frac{\beta}{2}-1}|\nabla V|^2\wv  ^{2-\beta} $ is substantial in \eqref{eq:strong-main}, $\Phi\asymp V^\frac{\beta}{2}$ and
\begin{equation*}
 V^{\frac{\beta}{2}-1}|\nabla V|^2\wv  ^{2-\beta} =V^{\frac{\beta}{2}-1}|\nabla V|^2\wv  ^{2-\beta}
 \gtrsim V^{\frac{\beta}{2}-1-\frac{2(1-\alpha)}{\alpha}} \asymp \Phi^{-\sigma} \asymp \kappa^\sigma\left(\ln (e+m)\right)^{-\sigma}.
\end{equation*}
since $\ln (e+m)\asymp \Phi$. If $\zeta(\Theta)>\frac{1}{2}$, then actually $\Theta\ge c>0$ and hence the same energy estimate as for $\Theta\ge2$ holds.

Gathering everything ends the proof of \Cref{thm:Lyap-alphalt1-betalt2}.
\end{proof}


\section{The case \texorpdfstring{$\alpha<1$}{} and \texorpdfstring{$\beta=2$}{}}\label{sec:Lyapunov5}

In this section, we revisit Cao's case discussed in \cite{Cao2019}, which is $0<\alpha<1$, $\beta=2$, with the light given by the previous construction. In this case the velocity cell corrector is exactly $\chi(v)=v$ as one can check immediately, 
\begin{equation*}
  \sfL^\star=\Delta_v-v\cdot\nabla_v, \qquad   \chi(v)=v,
  \qquad
  \sfL^\star\chi=-v,
\end{equation*}
so the construction is neater. 

Consequently, the corrected spatial phase used above reduces to
\begin{equation}\label{caseb2:eq:def-F}
  F(x,v):=V(x)+\nabla V(x)\cdot v.
\end{equation}

We fix once and for all
\begin{equation}\label{caseb2:eq:a-omega}
  a:=\frac18,
  \qquad
  \omega:=\frac1{32},
\end{equation}
and define
\begin{equation}\label{caseb2:eq:def-Theta}
  \Theta(x,v):=\frac{\wv }{E(x,v)^a}.
\end{equation}
Keeping the same notations as in the previous sections for the mollifier $\zeta$, $\Phi$ and $m$, the result proved in this section is the following endpoint version of the weak Lyapunov estimate.

\begin{theorem}\label{caseb2:thm:beta2}
There exist $\Phi_0\geq1$, $\kappa_0>0$, $R>0$, and constants $c,C>0$, depending only on $d$, $\alpha$, and the fixed profile $\zeta$, such that, for every $0<\kappa\leq\kappa_0$, the weight $m$ is positive, coercive, has compact sublevel sets, and satisfies
\begin{equation}\label{caseb2:eq:main-Lyapunov}
  \cL^\star m
  \leq C\1_{\{E\leq R\}}
  -c\,\kappa^{1+\frac{2(1-\alpha)}{\alpha}}
   \frac{m}{(\ln(e+m))^{\frac{2(1-\alpha)}{\alpha}}}.
\end{equation}
\end{theorem}

We shall only stress the very few differences that happen to appear compared to the proofs in the previous section. 

\subsection{Elementary estimates for $\Theta$}

We first record the elementary geometric facts needed in the interpolation region.  In what follows, the constants denoted by $c$ and $C$ may change from one line to the next, but depend only on $d$ and $\alpha$.

\begin{lemma}\label{caseb2:lem:infoTheta}
There is $R_\Theta>0$ such that, whenever $E \geq R_\Theta$,
\begin{enumerate}
    \item if $\Theta\le 2$ then $E\asymp V$. 
\item whenever $\Theta\le 2$ and $a<\frac12$,
\begin{equation*}
|\nabla_v\Theta|\lesssim E^{-a}, \qquad  |\nabla_v^2\Theta|\lesssim \Theta^{-1}E^{-2a}.    
\end{equation*}
\item Whenever $E^{-\omega}\le \Theta \le2$,
\begin{equation*}
 \cLs \Theta\le- \frac{\Theta}{2}.
\end{equation*}
\end{enumerate}   
\end{lemma}
Since the proof of this lemma is the same as the proof of \Cref{lem:infoTheta}, we do not reproduce it. The special case $\beta=2$ also gives an exact and particularly simple expression for the difference between the two phases.

\begin{lemma}\label{caseb2:lem:E-F}
There exist $R_1>0$ and constants $c_D,C_D>0$ such that, whenever $E\geq R_1$ and $E^{-\omega}\leq\Theta\leq2$ holds,
\begin{equation}\label{caseb2:eq:E-F-comparable}
  c_D\Theta^2E^{2a}
  \leq E-F
  \leq C_D\Theta^2E^{2a},
\end{equation}
and
\begin{equation}\label{caseb2:eq:grad-E-F}
  |\nabla_v(E-F)|\leq C_D\Theta E^a.
\end{equation}
\end{lemma}
Here, the proof is much more direct. 
\begin{proof}[{\bf Proof of \Cref{caseb2:lem:E-F}}]
Since
\begin{equation*}
  E-F=\frac{|v|^2}{2}-\nabla V\cdot v,
\end{equation*}
and on $E^{-\omega}\leq\Theta\leq2$, $\wv\geq E^{a-\omega}\to\infty$, 
\begin{equation*}
  \frac{|\nabla V|}{|v|}
  \lesssim C E^{1-\frac{1}{\alpha}-a+\omega}
\end{equation*}
is small. After increasing $R_1$, we may therefore assume
  $|\nabla V\cdot v|\leq\frac14|v|^2$. It follows that
\begin{equation*}
  \frac14|v|^2\leq E-F\leq\frac34|v|^2,
\end{equation*}
which is exactly \eqref{caseb2:eq:E-F-comparable}.  Finally,
\begin{equation*}
  \nabla_v(E-F)=v-\nabla V,
\end{equation*}
and the preceding estimates give \eqref{caseb2:eq:grad-E-F}.
\end{proof}

\subsection{Dissipation of $E$ and $F$}

The two basic phases $F$ and $E$ have completely explicit drifts at $\beta=2$ and the estimates are simpler.

\begin{lemma}\label{caseb2:lem:base-dissipation}
There exist $R_2>0$ and $\kappa_1>0$ such that, for every $0<\kappa\leq\kappa_1$, the following estimates hold.
\begin{enumerate}[label=\textup{(\roman*)},leftmargin=*]
\item If $E\geq R_2$ and $\Theta\leq2$, then
\begin{equation}\label{caseb2:eq:F-dissipation}
  \cL^\star F+4\kappa|\nabla_vF|^2
  \leq-\frac12|\nabla V|^2.
\end{equation}
\item If $E\geq R_2$ and $\Theta\geq E^{-\omega}$, then
\begin{equation}\label{caseb2:eq:E-dissipation}
  \cL^\star E+4\kappa|\nabla_vE|^2
  \leq-\frac12|v|^2.
\end{equation}
\end{enumerate}
\end{lemma}

\begin{proof}[{\bf Proof of \Cref{caseb2:lem:base-dissipation}}]
Since $F=V+\nabla V\cdot v$ and $\sfL^\star v=-v$, the first-order terms cancel exactly:
\begin{equation}\label{caseb2:eq:LF-exact}
  \cL^\star F
  =v^\top D_x^2V\,v-|\nabla V|^2,
  \qquad
  \nabla_vF=\nabla V.
\end{equation}
For the present potential,
\begin{equation*}
  D_x^2V
  =\wx^{\alpha-2}I
   +(\alpha-2)\wx^{\alpha-4}x\otimes x,
  \qquad
  \|D_x^2V\|\leq3\wx^{\alpha-2}.
\end{equation*}
Also,
\begin{equation*}
  |\nabla V|^2
  =\wx^{2\alpha-2}\bigl(1-\wx^{-2}\bigr).
\end{equation*}
On $\{E\geq R_2,\ \Theta\leq2\}$, Lemma \ref{caseb2:lem:infoTheta} gives $E\asymp V$, and therefore
\begin{equation*}
  \frac{|v^\top D_x^2V\,v|}{|\nabla V|^2}
  \leq C_\alpha\frac{|v|^2}{V}
  \leq C_\alpha E^{2a-1}=o(1).
\end{equation*}
Thus, after increasing $R_2$,
\begin{equation*}
  |v^\top D_x^2V\,v|\leq\frac14|\nabla V|^2.
\end{equation*}
Taking $\kappa_1\leq\frac{1}{16}$ in \eqref{caseb2:eq:LF-exact} proves \eqref{caseb2:eq:F-dissipation}.

For the energy,
\begin{equation}\label{caseb2:eq:LE-exact}
  \cL^\star E=d-|v|^2,
  \qquad
  \nabla_vE=v.
\end{equation}
Hence, for $\kappa\leq\frac{1}{16}$,
\begin{equation*}
  \cL^\star E+4\kappa|\nabla_vE|^2
  =d-(1-4\kappa)|v|^2
  \leq d-\frac34|v|^2.
\end{equation*}
If $\Theta\geq E^{-\omega}$, then $\wv \geq E^{a-\omega}$, so $|v|\to\infty$ uniformly as $E\to\infty$.  Increasing $R_2$ once more yields \eqref{caseb2:eq:E-dissipation}.
\end{proof}

\subsection{Estimates on $\Phi$}

We now estimate the full phase.  The following proposition is the core of the construction.

\begin{proposition}\label{caseb2:prop:phase-drift}
There exist $\kappa_0>0$, $R_3>0$, and $c_0>0$ such that, for every $0<\kappa\leq\kappa_0$,
\begin{equation}\label{caseb2:eq:phase-drift}
  \cL^\star\Phi+2\kappa|\nabla_v\Phi|^2
  \leq-c_0\left[
    \bigl(1-\zeta(\Theta)\bigr)|\nabla V|^2
    +\zeta(\Theta)|v|^2
  \right]
\end{equation}
whenever $E\geq R_3$.
\end{proposition}

The only difference in the proof compared to the proof of \Cref{prop:LyapPhi} is the estimate of the squared term in the zone $E^{-\omega}\leq\Theta\leq1$. It should now be estimated as follows.  \begin{equation*}
  4\kappa|\zeta'(\Theta)(E-F)\nabla_v\Theta|^2
  \leq C\kappa\zeta(\Theta)^2E^{2a}.
\end{equation*}
Since
\begin{equation*}
  \sup_{0<s\leq1}s^{-2}e^{-\frac{1}{s}}<\infty,
\end{equation*}
one has $\zeta(s)\leq Cs^2$ on $(0,1]$, and therefore
\begin{equation*}
  4\kappa|\zeta'(\Theta)E-F\nabla_v\Theta|^2
  \leq C\kappa\zeta(\Theta)\Theta^2E^{2a}
  \leq C\kappa\zeta(\Theta)|v|^2.
\end{equation*}
After increasing $R_3$ and then choosing $\kappa_0$ small enough, it is absorbed by the energy dissipation. The rest of the proof is the same, we do not reproduce it. 

The end of the proof of \Cref{caseb2:thm:beta2} being similar to the end of the proof of \Cref{thm:Lyap-alphalt1-betalt2}, we shall not reproduce the arguments. 

\bibliographystyle{plain} 
\bibliography{library} 

\end{document}